\documentclass[a4paper,oneside,10.5pt]{article}%
\usepackage{amsmath}
\usepackage{amsfonts}
\usepackage{amssymb}
\usepackage{graphicx}%
\providecommand{\U}[1]{\protect\rule{.1in}{.1in}}

\newtheorem{theorem}{Theorem}[section]

\newtheorem{definition}[theorem]{Definition}
\newtheorem{assumption}[theorem]{Assumption}

\newtheorem{lemma}[theorem]{Lemma}

\newtheorem{remark}[theorem]{Remark}

\newenvironment{proof}[1][Proof]{\noindent \textbf{#1.} }{\  \rule{0.5em}{0.5em}}
\numberwithin{equation}{section}

\begin{document}

\title{Maximum principle for stochastic optimal control problem of finite state
forward-backward stochastic difference systems}
\author{Shaolin Ji\thanks{Zhongtai Securities Institute for Financial Studies,
Shandong University, Jinan, Shandong 250100, PR China. jsl@sdu.edu.cn.
Research supported by NSF (No. 11571203).}
\and Haodong Liu\thanks{Zhongtai Securities Institute for Financial Studies,
Shandong University, Jinan, Shandong 250100, PR China. (Corresponding
author).}}
\maketitle

\textbf{Abstract}: In this paper, we study the maximum principle for
stochastic optimal control problems of forward-backward stochastic difference
systems (FBS$\Delta$Ss) where the uncertainty is modeled by a discrete time,
finite state process, rather than white noises. Two types of FBS$\Delta$Ss are
investigated. The first one is described by a partially coupled
{forward-backward stochastic difference equation (}FBS$\Delta$E) and the
second one is described by a fully coupled FBS$\Delta$E. By adopting an
appropriate representation of the product rule and an appropriate formulation
of the {backward stochastic difference equation (}BS$\Delta$E), we deduce the
adjoint difference equation. Finally, the maximum principle for this optimal
control problem with the control domain being convex is established.

{\textbf{Keywords}: backward stochastic difference equations; forward-backward
stochastic difference equations; monotone condition; stochastic optimal
control; }maximum principle

\addcontentsline{toc}{section}{\hspace*{1.8em}Abstract}

\section{Introduction}

The Maximum Principle is one of the important approaches in solving the
optimal control problems. A lot of work has been done on the Maximum Principle
for stochastic system. See, for example, Bensoussan \cite{b82}, Bismut
\cite{b78}, Kushner \cite{k72}, Peng \cite{p90}. Peng also firstly studied one
kind of forward-backward stochastic control system (FBSCS) in \cite{p93} and
obtained the maximum principle for this kind of control system with control
domain being convex. The FBSCSs have wide applications in many fields. As the
stochastic differential recursive utility, which is a generalization of a
standard additive utility, can be regarded as a solution of a {backward
stochastic differential equation (}BSDE). The recursive utility optimization
problem can be described by a optimization problem for a FBSCS (see
\cite{ss99}). Besides, in the dynamic principal-agent problem with
unobservable states and actions, the principal's problem can be formulated as
a partial information optimal control problem of a FBSCS (see \cite{w09}). We
refer to \cite{dz99}, \cite{hjx18}, \cite{hjx18+}, \cite{lz01},\ \cite{sz13},
\cite{x95}, \cite{zs11}\ for other works on optimization problems for FBSCSs.

In this paper, we will discuss the Maximum Principle for optimal control of
discrete time systems described by forward-backward stochastic difference
equations (FBS$\Delta$Es). To the best of our knowledge, there are few results
on such optimization control problems. In fact, the discrete time control
systems are of great value in practice. For example, the digital control can
be formulated as discrete time control problems, where the sampled data is
obtained at discrete instants of time. Besides, the forward-backward
stochastic difference system (FBS$\Delta$S) can be used for modeling in
financial markets. For example, the solution to the backward stochastic
difference equation (BS$\Delta$E) can be used to construct time-consistent
nonlinear expectations (see \cite{ce10+}, \cite{ce11}) and be used for pricing
in the financial markets (see \cite{bcc15}). However, the formulation of
BS$\Delta$E is quite different from its continuous time counterpart. Many
works are devoted to the study of BS$\Delta$Es (see, e.g. \cite{bcc15},
\cite{ce10+}, \cite{ce11}, \cite{s10}). Based on the driving process, there
are mainly two types of formulations of BS$\Delta$Es. One is driving by a
finite state process which takes values from the basis vectors (as in
\cite{ce10+}) and the other is driving by a martingale with independent
increments (as in \cite{bcc15}). For the former framework, the researchers in
\cite{ce10+} obtained the discrete time version of martingale representation
theorem and establish the solvability result of BS$\Delta$E with the
uniqueness of $Z$ under a new kind of equivalence relation.\ Further works
about the applications of the finite state framework can be seen in
\cite{em11}, \cite{m10}, \cite{ly14}. In this paper, we adopt the first type
of formulation to investigate the optimization problems for FBS$\Delta$Ss.

In this paper, we study two stochastic optimal control problems. The Problem 1
involves a partially coupled FBS$\Delta$E (\ref{c5_eq_m_pc_state_eq}). In more
details, the coefficients $b$ and $\sigma$ of the forward equation do not
contain the solution $(Y,Z)$ of the backward equation. The state equation of
Problem 2 is described by a fully coupled FBS$\Delta$E (\ref{f_c_c_2_state_eq}).

The optimal control problem is to find the optimal control $u\in\mathcal{U}$,
such that the optimal control and the corresponding state trajectory can
minimize the cost functional $J\left(  u\left(  \cdot\right)  \right)  $. In
this paper, we assume the control domain is convex. By making the perturbation
of the optimal control at a fixed time point, we obtain the maximum principle
for problem 1 and 2.

To build the maximum principle, the key step is to find the adjoint variables
which can be applied to deduce the variational inequality. In \cite{lz15}, the
authors studied the maximum principle for a discrete time stochastic optimal
control problem in which the state equation is only governed by a forward
stochastic difference equation. By applying the Riesz representation theorem,
they explicitly obtained the adjoint variables and establish the maximum
principle. But to solve our problems, we need to construct the adjoint
difference equations since generally the adjoint variables can not be obtained
explicitly for our case. To construct the adjoint equations in our discrete
time framework, the techniques which are adopted for the continuous time
framework as in \cite{p90,p93} are not applicable. In this paper, we propose
two techniques to deduce the adjoint difference equations. The first one is
that we choose the following product rule:%
\[
\Delta\left\langle X_{t},Y_{t}\right\rangle =\left\langle X_{t+1},\Delta
Y_{t}\right\rangle +\left\langle \Delta X_{t},Y_{t}\right\rangle
\]
where $X_{t}$ (resp. $Y_{t}$) subjects to a forward (resp. backward)
stochastic difference equation. The second one is that the BS$\Delta$E should
be formulated as in (\ref{bsde_form1}). In other words, the generator $f$ of
the BS$\Delta$E (\ref{bsde_form1}) depends on time $t+1$. It is worth pointing
out that this kind of formulation is just the formulation of the adjoint
equations for stochastic optimal control problems (see \cite{lz15} for tha
classical case). Based on these two techniques, we can deduce the adjoint
difference equations. The readers may refer to Remark \ref{re-product-rule}
for more details.

Besides, the second difficulty is in the finite state space case. Since the
uniqueness of the variable $Z$ is not defined in the normal sense, the norm of
the variable should be redefined. In \cite{ce10+}, Cohen and Elliott defined a
seminorm of $Z_{t}$ through the term $Z_{t}M_{t+1}$. However, since the
It\^{o} isometry cannot work in the discrete time case and the martingale
difference process $M_{t}$ depends on the past, the relation between the norm
defined by $Z_{t}$ itself and the norm defined by $Z_{t}M_{t+1}$ is not clear.
So it makes estimating the diffusion term of the variation equations quite
difficult. In this paper, we propose a new definition of the norm for the
variable $Z_{t}$ in the diffusion term and prove the relation between this
norm of $Z_{t}$ and the seminorm defined by $Z_{t}M_{t+1}$. With this
relation, we can derive the estimation of the solutions to the stochastic
difference equations in the discrete time finite state space framework.

The remainder of this paper is organized as follows. In section 2, two types
of the controlled FBS$\Delta$Ss are formulated. We deduce the maximum
principle for the partially coupled controlled FBS$\Delta$S in section 3.
Finally, we establish the maximum principle for the fully coupled controlled
FBS$\Delta$S in section 4.

\section{Preliminaries and model formulation}

Let $T$ be a deterministic terminal time and $\mathcal{T}:=\left\{
0,1,...,T\right\}  $. Following \cite{ce10+}, we consider an underlying
discrete time, finite state process $W$ which takes values in the standard
basis vectors of $\mathbb{R}^{d}$, where $d$ is the number of states of the
process $W$. In more detail, for each $t\in\mathcal{T}$, $W_{t}\in\left\{
e_{1},e_{2},...e_{d}\right\}  $ where $e_{i}=\left(
0,0,...,0,1,0,...,0\right)  ^{\ast}\in\mathbb{R}^{d}$ and $\left[
\cdot\right]  ^{\ast}$ denotes vector transposition.

Consider a filtered probability space $\left(  \Omega,\mathcal{F},\left\{
\mathcal{F}_{t}\right\}  _{0\leq t\leq T},P\right)  $, where $\mathcal{F}_{t}$
is the completion of the $\sigma$-algebra generated by the process $W$ up to
time $t$ and $\mathcal{F}=\mathcal{F}_{T}$. Denote by $L\left(  \mathcal{F}%
_{t};\mathbb{R}^{n\times d}\right)  $ the set of all $\mathcal{F}_{t}-$adapted
random variable $X_{t}$ taking values in $\mathbb{R}^{n\times d}$ and by
$\mathcal{M}\left(  0,t;\mathbb{R}^{n\times d}\right)  $ the set of all
$\mathcal{F}_{t}$-adapted process $X$ taking values in $\mathbb{R}^{n\times
d}$ with the norm $\left\Vert \cdot\right\Vert $ defined by $\left\Vert
X\right\Vert =\left(  \mathbb{E}\left[  \sum_{s=0}^{t}\left\vert
X_{s}\right\vert ^{2}\right]  \right)  ^{\frac{1}{2}}$.

For simplicity, we suppose the process $W$ satisfies the following assumption.
Note that in the following, an inequality on a vector quantity is to hold componentwise.

\begin{assumption}
\label{general_assumption}For any $t\in\left\{  0,1,2,...,T-1\right\}  $, any
$\omega\in\Omega$, $\mathbb{E}\left[  W_{t+1}|\mathcal{F}_{t}\right]  \left(
\omega\right)  >0.$
\end{assumption}

The above assumption means that the probability of every possible path of $W$
on $\left\{  0,1,2,...,T\right\}  $ is strictly positive. Hence under this
assumption, the conception "$P-$almost surely" in the following statements can
be changed to "for every $\omega$". In fact, this assumption is given just for
simple statements. Without this assumption, the proof ideas are the same, but
the statements are more sophisticated. We set $\mathbb{E}\left[
W_{t+1}|\mathcal{F}_{t}\right]  =\left(  P_{t}^{1},P_{t}^{2},...,P_{t}%
^{N}\right)  ^{\ast}$.

Define%
\[
M_{t}=W_{t}-\mathbb{E}\left[  W_{t}|\mathcal{F}_{t-1}\right]  ,t=1,...,T.
\]
$M$ is a martingale difference process taking values in $\mathbb{R}^{d}$. The
following equivalence relations given in \cite{ce10+} will be used in the following.

\begin{definition}
For two $\mathcal{F}_{t}$-measurable random variables $Z_{t}$ and
$\widetilde{Z}_{t}$, we define $Z_{t}\thicksim_{M_{t+1}}\widetilde{Z}_{t}$, if
$Z_{t}M_{t+1}=\widetilde{Z}_{t}M_{t+1},$ $P-a.s.;$

For two adapted processes $Z$ and $\widetilde{Z}$, we define $Z\thicksim
_{M}\widetilde{Z}$, if $Z_{t}M_{t+1}=\widetilde{Z}_{t}M_{t+1},$ $P-a.s.$ for
any $t\in\left\{  0,1,2,...,T-1\right\}  .$
\end{definition}

For a $\mathcal{F}_{t}$-adapted process $X$, define the difference operator
$\Delta$ as $\Delta X_{t}=X_{t+1}-X_{t}$. Consider the following backward
stochastic difference equation (BS$\Delta$E):%

\begin{equation}
\left\{
\begin{array}
[c]{rcl}%
\Delta Y_{t} & = & -f\left(  \omega,t+1,Y_{t+1},Z_{t+1}\right)  +Z_{t}%
M_{t+1},\\
Y_{T} & = & \eta.
\end{array}
\right.  \label{bsde_form1}%
\end{equation}
where $\eta\in L\left(  \mathcal{F}_{T};\mathbb{R}^{n}\right)  $ and
$f:\Omega\times\left\{  1,2,...,T\right\}  \times\mathbb{R}^{n}\times
\mathbb{R}^{n\times d}\longmapsto\mathbb{R}^{n}$ is $\mathcal{F}_{t}$-adapted mapping.

\begin{assumption}
\label{generator_assumption}A1. For any $y\in\mathbb{R}^{n}$, $t\in\left\{
1,2,...,T-1\right\}  $,\ $\omega\in\Omega$, and $Z^{1},$ $Z^{2}\in
\mathcal{M}\left(  0,T-1;\mathbb{R}^{n\times d}\right)  $, if $Z^{1}%
\thicksim_{M}Z^{2}$, then%
\[
f\left(  \omega,t,y,Z_{t}^{1}\right)  =f\left(  \omega,t,y,Z_{t}^{2}\right)
.
\]

A2. The function $f\left(  t,y,z\right)  $ is independent of $z$ at $t=T$.
\end{assumption}

We have the following existence and uniqueness theorem of BS$\Delta$E
(\ref{bsde_form1}) in \cite{ji-liu}.

\begin{theorem}
\label{bsde_result_l}Suppose that Assumption (\ref{generator_assumption})
holds. Then for any terminal condition $\eta\in L\left(  \mathcal{F}%
_{T};\mathbb{R}^{n}\right)  $, BS$\Delta$E (\ref{bsde_form1}) has a unique
adapted solution $\left(  Y,Z\right)  $. Here the uniqueness for $Y$ is in the
sense of indistinguishability and for $Z$ is in the sense of $\thicksim_{M}$ equivalence.
\end{theorem}

We define the $d\times\left(  d-1\right)  $ matrix $\widetilde{I}=%
\begin{pmatrix}
I_{d-1} & -\mathbf{1}_{d-1}%
\end{pmatrix}
^{\ast},$ where $I_{d-1}$ is $\left(  d-1\right)  $-dimensional identity
matrix, $\mathbf{1}_{d-1}\mathbf{=}\left(  1,1,...,1\right)  ^{\ast}$ is
$\left(  d-1\right)  $-dimensional vector with every element being equal to
$1$. Then, we consider two types of controlled systems.

Problem 1 (partially coupled system):

The controlled system is%
\begin{equation}
\left\{
\begin{array}
[c]{rcl}%
\Delta X_{t} & = & b\left(  \omega,t,X_{t},u_{t}\right)  +\sum\limits_{i=1}%
^{m}e_{i}\cdot\sigma_{i}\left(  \omega,t,X_{t},u_{t}\right)  M_{t+1},\\
\Delta Y_{t} & = & -f\left(  \omega,t+1,X_{t+1},Y_{t+1},Z_{t+1}\widetilde{I}%
,u_{t+1}\right)  +Z_{t}M_{t+1},\\
X_{0} & = & x_{0},\\
Y_{T} & = & y_{T}%
\end{array}
\right.  \label{c5_eq_m_pc_state_eq}%
\end{equation}

and the cost functional is%
\begin{equation}
J\left(  u\left(  \cdot\right)  \right)  =\mathbb{E}\sum_{t=0}^{T-1}l\left(
\omega,t,X_{t},Y_{t},Z_{t}\widetilde{I},u_{t}\right)  +h\left(  \omega
,X_{T}\right)  \label{c5_eq_m_pc_cost_eq}%
\end{equation}
where%
\begin{align*}
b\left(  \omega,t,x,u\right)   &  :\Omega\times\left\{  0,1,...,T-1\right\}
\times\mathbb{R}^{m}\times\mathbb{R}^{r}\mathbb{\rightarrow R}^{m},\\
\sigma_{i}\left(  \omega,t,x,u\right)   &  :\Omega\times\left\{
0,1,...,T-1\right\}  \times\mathbb{R}^{m}\times\mathbb{R}^{r}%
\mathbb{\rightarrow R}^{1\times d},\\
f\left(  \omega,t,x,y,\widetilde{z},u\right)   &  :\Omega\times\left\{
1,2,...,T\right\}  \times\mathbb{R}^{m}\times\mathbb{R}^{n}\times
\mathbb{R}^{n\times\left(  d-1\right)  }\times\mathbb{R}^{r}%
\mathbb{\rightarrow R}^{m},\\
l\left(  \omega,t,x,y,\widetilde{z},u\right)   &  :\Omega\times\left\{
0,1,...,T-1\right\}  \times\mathbb{R}^{m}\times\mathbb{R}^{n}\times
\mathbb{R}^{n\times\left(  d-1\right)  }\times\mathbb{R}^{r}%
\mathbb{\rightarrow R},\\
h\left(  \omega,x\right)   &  :\Omega\times\mathbb{R}^{m}\mathbb{\rightarrow
R}.
\end{align*}

Problem 2 (fully coupled system):

The controlled system is:%
\begin{equation}
\left\{
\begin{array}
[c]{rcl}%
\Delta X_{t} & = & b\left(  \omega,t,X_{t},Y_{t},Z_{t}\widetilde{I}%
,u_{t}\right)  +\sigma\left(  \omega,t,X_{t},Y_{t},Z_{t}\widetilde{I}%
,u_{t}\right)  M_{t+1},\\
\Delta Y_{t} & = & -f\left(  \omega,t+1,X_{t+1},Y_{t+1},Z_{t+1}\widetilde{I}%
,u_{t+1}\right)  +Z_{t}M_{t+1},\\
X_{0} & = & x_{0},\\
Y_{T} & = & y_{T},
\end{array}
\right.  \label{f_c_c_2_state_eq}%
\end{equation}

and the cost functional is%
\begin{equation}
J\left(  u\left(  \cdot\right)  \right)  =\mathbb{E}\sum_{t=0}^{T-1}l\left(
\omega,t,X_{t},Y_{t},Z_{t}\widetilde{I},u_{t}\right)  +h\left(  \omega
,X_{T}\right)  \label{f_c_c_2_cost_eq}%
\end{equation}
where%
\begin{align*}
b\left(  \omega,t,x,y,\widetilde{z},u\right)   &  :\Omega\times\left\{
0,1,...,T-1\right\}  \times\mathbb{R}\times\mathbb{R}\times\mathbb{\mathbb{R}%
}^{1\times\left(  d-1\right)  }\times\mathbb{R}^{r}\mathbb{\rightarrow R},\\
\sigma\left(  \omega,t,x,y,\widetilde{z},u\right)   &  :\Omega\times\left\{
0,1,...,T-1\right\}  \times\mathbb{R}\times\mathbb{R}\times\mathbb{\mathbb{R}%
}^{1\times\left(  d-1\right)  }\times\mathbb{R}^{r}\mathbb{\rightarrow
R}^{1\times d},\\
f\left(  \omega,t,x,y,\widetilde{z},u\right)   &  :\Omega\times\left\{
1,2,...,T\right\}  \times\mathbb{R}\times\mathbb{R}\times\mathbb{\mathbb{R}%
}^{1\times\left(  d-1\right)  }\times\mathbb{R}^{r}\mathbb{\rightarrow R},\\
l\left(  \omega,t,x,y,\widetilde{z},u\right)   &  :\Omega\times\left\{
0,1,...,T-1\right\}  \times\mathbb{R}\times\mathbb{R}\times\mathbb{\mathbb{R}%
}^{1\times\left(  d-1\right)  }\times\mathbb{R}^{r}\mathbb{\rightarrow R},\\
h\left(  \omega,x\right)   &  :\Omega\times\mathbb{R\rightarrow R}.
\end{align*}

Let $\left\{  U_{t}\right\}  _{t\in\left\{  0,1,...,T\right\}  }$ be a
sequence of nonempty convex subset of $\mathbb{R}^{r}$. We denote the set of
admissible controls $\mathcal{U}$ by $\mathcal{U}=\left\{  u\left(
\cdot\right)  \in\mathcal{M}\left(  0,T;\mathbb{R}^{r}\right)  |u\left(
t\right)  \in U_{t}\right\}  .$ It can be seen that in Problem 1, $b$ and
$\sigma$ do not contain the solution $(Y,Z)$ of the backward equation. This
kind of FBS$\Delta$E is called the partially coupled FBS$\Delta$E. Meanwhile,
the system in Problem 2 is called the fully coupled FBS$\Delta$E.

The optimal control problem is to find the optimal control $u\in\mathcal{U}$,
such that the optimal control and the corresponding state trajectory can
minimize the cost functional $J\left(  u\left(  \cdot\right)  \right)  $. In
this paper, we assume the control domain is convex.

\begin{remark}
The cost functional in \cite{p93} consists of three parts: the running cost
functional, the terminal cost functional of $X_{T}$, the initial cost
functional of $Y_{0}$. In our formulation, if we take $l\left(  \omega
,0,X_{0},Y_{0},Z_{0},u_{0}\right)  =\gamma\left(  \omega,Y_{0}\right)  $, then
the cost functional (\ref{f_c_c_2_cost_eq}) for our discrete time framework
can be reduced to the cost functional in \cite{p93} formally.
\end{remark}

For controlled system (\ref{c5_eq_m_pc_state_eq})-(\ref{c5_eq_m_pc_cost_eq}),
we assume that:

\begin{assumption}
\label{c5_assumption_m_pc}For $\varphi=b$, $\sigma_{i}\widetilde{I}$, $f$,
$l$, $h$,

\begin{enumerate}
\item $\varphi$\ is an adapted map, i.e. for any\ $\left(  x,y,\widetilde{z}%
,u\right)  \in\mathbb{R}^{m}\times\mathbb{R}^{n}\times\mathbb{R}%
^{n\times\left(  d-1\right)  }\times\mathbb{R}^{r}$, $\varphi\left(
\cdot,\cdot,x,y,\widetilde{z},u\right)  $ is $\left\{  \mathcal{F}%
_{t}\right\}  $-adapted process.

\item for any $t\in\left\{  0,1,...,T\right\}  $ and $\omega\in\Omega$,
$\varphi\left(  \omega,t,\cdot,\cdot,\cdot,\cdot\right)  \,$is continuously
differentiable with respect to $x,y,\widetilde{z},u$, and $\varphi_{x}%
,\varphi_{y},\varphi_{\widetilde{z}_{i}},\varphi_{u}$ are uniformly bounded.
Also, for $t=T$, $f$ is independent of $\widetilde{z}$ at time $T$.
\end{enumerate}
\end{assumption}

Set%
\begin{align*}
\lambda &  =\left(  x,y,z\right)  ,\\
A\left(  t,\lambda;u\right)   &  =\left(  -f\left(  t,\lambda;u\right)
,b\left(  t,\lambda;u\right)  ,\sigma\left(  t,\lambda;u\right)
\mathbb{E}\left[  M_{t+1}M_{t+1}^{\ast}|\mathcal{F}_{t}\right]  \right)
\end{align*}
and%
\begin{align*}
\left\vert \lambda\right\vert  &  =\left\vert x\right\vert +\left\vert
y\right\vert +\left\vert z\widetilde{I}\right\vert ,\\
\left\vert A\left(  t,\lambda\right)  \right\vert  &  =\left\vert f\left(
t,\lambda\right)  \right\vert +\left\vert b\left(  t,\lambda\right)
\right\vert +\left\vert \sigma\left(  t,\lambda\right)  \mathbb{E}\left[
M_{t+1}M_{t+1}^{\ast}|\mathcal{F}_{t}\right]  \right\vert .
\end{align*}

For controlled system (\ref{f_c_c_2_state_eq})-(\ref{f_c_c_2_cost_eq}), we
additionally assume that:

\begin{assumption}
\label{f_c_c_2_control_assumption}For any $u\in\mathcal{U}$, the coefficients
in (\ref{f_c_c_2_state_eq}) satisfy the following monotone conditions, i.e.
when $t\in\left\{  1,...,T-1\right\}  $,%
\begin{gather*}
\left\langle A\left(  t,\lambda_{1};u\right)  -A\left(  t,\lambda
_{2};u\right)  ,\lambda_{1}-\lambda_{2}\right\rangle \leq-\alpha\left\vert
\lambda_{1}-\lambda_{2}\right\vert ^{2},\\
\forall\lambda_{1},\lambda_{2}\in\mathbb{R}\times\mathbb{R}\times
\mathbb{\mathbb{R}}^{1\times d};
\end{gather*}

when $t=T$,%
\[
\left\langle -f\left(  T,x_{1},y,z\widetilde{I},u\right)  +f\left(
T,x_{2},y,z\widetilde{I},u\right)  ,x_{1}-x_{2}\right\rangle \leq
-\alpha\left\vert x_{1}-x_{2}\right\vert ^{2};
\]

when $t=0$,%
\[%
\begin{array}
[c]{cl}
& \left\langle b\left(  0,\lambda_{1};u\right)  -b\left(  0,\lambda
_{2};u\right)  ,y_{1}-y_{2}\right\rangle \\
& +\left\langle \left(  \sigma\left(  0,\lambda_{1};u\right)  -\sigma\left(
0,\lambda_{2};u\right)  \right)  \mathbb{E}\left[  M_{1}M_{1}^{\ast
}|\mathcal{F}_{0}\right]  ,z_{1}-z_{2}\right\rangle \\
\leq & -\alpha\left[  \left\vert y_{1}-y_{2}\right\vert ^{2}+\left\vert
\left(  z_{1}-z_{2}\right)  \widetilde{I}\right\vert ^{2}\right]  .
\end{array}
\]
where $\alpha$ is a given positive constant.
\end{assumption}

Besides, in the following, we formally denote $b\left(  T,x,y,z\widetilde{I}%
,u\right)  \equiv0$, $\sigma\left(  T,x,y,z\widetilde{I},u\right)  \equiv0$,
$l\left(  T,x,y,z\widetilde{I},u\right)  \equiv0$, $f\left(
0,x,y,z\widetilde{I},u\right)  \equiv0$.

\section{Maximum principle for the partially coupled FBS$\Delta$E system}

For any $u\in\mathcal{U}$, it is obvious that there exists a unique solution
$\left\{  X_{t}\right\}  _{t=0}^{T}\in\mathcal{M}\left(  0,T;\mathbb{R}%
^{m}\right)  $ to the forward stochastic difference equation in the system
(\ref{c5_eq_m_pc_state_eq}). According to Lemma 2.3 in \cite{ji-liu}, it can
be seen that $f$ satisfies Assumption (\ref{generator_assumption}). So given
$X$, by Theorem \ref{bsde_result_l}, the backward equation in the system
(\ref{c5_eq_m_pc_state_eq}) has a unique solution $\left(  Y,Z\right)  $.

Suppose that $\bar{u}=\left\{  \bar{u}_{t}\right\}  _{t=0}^{T}$ is the optimal
control of problem (\ref{c5_eq_m_pc_state_eq})-(\ref{c5_eq_m_pc_cost_eq}) and
$\left(  \bar{X},\bar{Y},\bar{Z}\right)  $ is the corresponding optimal
trajectory. For a fixed time $0\leq s\leq T$, choose any $\Delta v\in L\left(
\mathcal{F}_{s};\mathbb{R}^{r}\right)  $ such that $\bar{u}_{s}+\Delta v$
takes values in $U_{s}$. For any $\varepsilon\in\left[  0,1\right]  $,
construct the perturbed admissible control
\begin{equation}
u_{t}^{\varepsilon}=\left(  1-\delta_{ts}\right)  \bar{u}_{t}+\delta
_{ts}\left(  \bar{u}_{s}+\varepsilon\Delta v\right)  =\bar{u}_{t}+\delta
_{ts}\varepsilon\Delta v, \label{perturb-control}%
\end{equation}
where $\delta_{ts}=1$ for $t=s$, $\delta_{ts}=0$ for $t\neq s$ and
$t\in\left\{  0,1,...,T\right\}  $. Since $U_{s}$ is a convex set, $\left\{
u_{t}^{\varepsilon}\right\}  _{t=0}^{T}\in\mathcal{U}$ is an admissible
control. Let $\left(  X^{\varepsilon},Y^{\varepsilon},Z^{\varepsilon
},N^{\varepsilon}\right)  $ be the solution of (\ref{c5_eq_m_pc_state_eq})
corresponding to the control $u^{\varepsilon}$.

Set%
\begin{equation}%
\begin{array}
[c]{rclrcl}%
\bar{\varphi}\left(  t\right)  & = & \varphi\left(  t,\bar{X}_{t},\bar{Y}%
_{t},\bar{Z}_{t}\widetilde{I},\bar{u}_{t}\right)  , & \varphi^{\varepsilon
}\left(  t\right)  & = & \varphi\left(  t,X_{t}^{\varepsilon},Y_{t}%
^{\varepsilon},Z_{t}^{\varepsilon}\widetilde{I},u_{t}^{\varepsilon}\right)
,\\
\widetilde{\varphi}^{\varepsilon}\left(  t\right)  & = & \varphi\left(
t,\bar{X}_{t},\bar{Y}_{t},\bar{Z}_{t}\widetilde{I},u_{t}^{\varepsilon}\right)
, & \varphi_{\mu}\left(  t\right)  & = & \varphi_{\mu}\left(  t,\bar{X}%
_{t},\bar{Y}_{t},\bar{Z}_{t}\widetilde{I},\bar{u}_{t}\right)  ,
\end{array}
\label{coefficient_notations}%
\end{equation}
where $\varphi=b$, $\sigma_{i}$,\ $f$, $l$, $h$ and $\mu=x$, $y$, $z_{i}$ and
$u$.

Then, we have the following estimates.

\begin{lemma}
Under Assumption \ref{c5_assumption_m_pc}, we have%
\begin{equation}
\sup_{0\leq t\leq T}\mathbb{E}\left\vert X_{t}^{\varepsilon}-\bar{X}%
_{t}\right\vert ^{2}\leq C\varepsilon^{2}\mathbb{E}\left\vert \Delta
v\right\vert ^{2}. \label{c5_eq_m_pc_x_estimate_1}%
\end{equation}

\end{lemma}

\begin{proof}
In the following, the positive constant $C$ may change from lines to lines.

When $t=0,...,s$, $X_{t}^{\varepsilon}=\bar{X}_{t}$.

When $t=s+1$,
\[
X_{s+1}^{\varepsilon}-\bar{X}_{s+1}=\widetilde{b}^{\varepsilon}\left(
s\right)  -\overline{b}\left(  s\right)  +\sum\limits_{i=1}^{m}e_{i}%
\cdot\left[  \widetilde{\sigma_{i}}^{\varepsilon}\left(  s\right)
-\overline{\sigma}_{i}\left(  s\right)  \right]  M_{s+1}.
\]
Then,%
\[
\mathbb{E}\left\vert X_{s+1}^{\varepsilon}-\bar{X}_{s+1}\right\vert ^{2}%
\leq2\mathbb{E}\left[  \left\vert \widetilde{b}^{\varepsilon}\left(  s\right)
-\overline{b}\left(  s\right)  \right\vert ^{2}+\sum\limits_{i=1}%
^{m}\left\vert \left[  \widetilde{\sigma_{i}}^{\varepsilon}\left(  s\right)
-\overline{\sigma}_{i}\left(  s\right)  \right]  M_{s+1}\right\vert
^{2}\right]  .
\]
By the boundedness of $b_{u}$, we have
\[
\mathbb{E}\left[  \left\vert \widetilde{b}^{\varepsilon}\left(  s\right)
-\overline{b}\left(  s\right)  \right\vert ^{2}\right]  \leq C\mathbb{E}%
\left[  \left\vert u_{s}^{\varepsilon}-\bar{u}_{s}\right\vert ^{2}\right]
=C\varepsilon^{2}\mathbb{E}\left[  \left\vert \Delta v\right\vert ^{2}\right]
.
\]
By the Proposition 2.4\ in \cite{ji-liu} and boundedness of $\sigma
_{iu}\widetilde{I}$, we have%
\[%
\begin{array}
[c]{cl}
& \mathbb{E}\left\vert \left[  \widetilde{\sigma_{i}}^{\varepsilon}\left(
s\right)  -\overline{\sigma}_{i}\left(  s\right)  \right]  M_{s+1}\right\vert
^{2}\\
\leq & C\mathbb{E}\left[  \left\vert \left[  \widetilde{\sigma_{i}%
}^{\varepsilon}\left(  s\right)  -\overline{\sigma}_{i}\left(  s\right)
\right]  \widetilde{I}\right\vert ^{2}\right] \\
\leq & C\varepsilon^{2}\mathbb{E}\left[  \left\vert \Delta v\right\vert
^{2}\right]
\end{array}
\]
which leads to%
\[
\mathbb{E}\left\vert X_{s+1}^{\varepsilon}-\bar{X}_{s+1}\right\vert ^{2}\leq
C\varepsilon^{2}\mathbb{E}\left[  \left\vert \Delta v\right\vert ^{2}\right]
.
\]

When $t=s+2,...,T$,%
\[%
\begin{array}
[c]{cl}
& \mathbb{E}\left\vert X_{t}^{\varepsilon}-\bar{X}_{t}\right\vert ^{2}\\
\leq & 2\mathbb{E}\left[  \left\vert b\left(  t-1,X_{t-1}^{\varepsilon}%
,\bar{u}_{t-1}\right)  -b\left(  t-1,\bar{X}_{t-1},\bar{u}_{t-1}\right)
\right\vert ^{2}\right. \\
& +\left.  \sum\limits_{i=1}^{m}\left\vert \left[  \sigma_{i}\left(
t-1,X_{t-1}^{\varepsilon},\bar{u}_{t-1}\right)  -\sigma_{i}\left(  t-1,\bar
{X}_{t-1},\bar{u}_{t-1}\right)  \right]  M_{t}\right\vert ^{2}\right]  .
\end{array}
\]
Due to the boundedness of $b_{x}$, $\sigma_{ix}\widetilde{I}$, combined with
the Proposition 2.4,\ we obtain $\mathbb{E}\left\vert X_{t}^{\varepsilon}%
-\bar{X}_{t}\right\vert ^{2}\leq C\mathbb{E}\left[  \left\vert X_{t-1}%
^{\varepsilon}-\bar{X}_{t-1}\right\vert ^{2}\right]  $. Thus, by induction we
prove the result.
\end{proof}

Let $\xi=\left\{  \xi_{t}\right\}  _{t=0}^{T}$ be the solution to the
following difference equation,%
\begin{equation}
\left\{
\begin{array}
[c]{rcl}%
\Delta\xi_{t} & = & b_{x}\left(  t\right)  \xi_{t}+\delta_{ts}b_{u}\left(
t\right)  \varepsilon\Delta v+\sum\limits_{i=1}^{m}e_{i}\cdot\left[  \xi
_{t}^{\ast}\sigma_{ix}\left(  t\right)  +\delta_{ts}\varepsilon\Delta v^{\ast
}\sigma_{iu}\left(  t\right)  \right]  M_{t+1},\\
\xi_{0} & = & 0.
\end{array}
\right.  \label{variational eq-xi}%
\end{equation}

It is easy to check that%
\begin{equation}
\sup_{0\leq t\leq T}\mathbb{E}\left\vert \xi_{t}\right\vert ^{2}\leq
C\varepsilon^{2}\mathbb{E}\left\vert \Delta v\right\vert ^{2}.
\label{c5_eq_m_pc_xi_estimate}%
\end{equation}

and we have the following result:

\begin{lemma}
\label{c5_lemma_m_pc_x_approxi}Under Assumption \ref{c5_assumption_m_pc}, we
have%
\[
\sup_{0\leq t\leq T}\mathbb{E}\left\vert X_{t}^{\varepsilon}-\bar{X}_{t}%
-\xi_{t}\right\vert ^{2}=o\left(  \varepsilon^{2}\right)  .
\]

\end{lemma}

\begin{proof}
When $t=0,...,s$, $X_{t}^{\varepsilon}=\bar{X}_{t}$ and $\xi_{t}=0$ which lead
to $X_{t}^{\varepsilon}-\bar{X}_{t}-\xi_{t}=0.$

When $t=s+1$,%
\[%
\begin{array}
[c]{cl}
& X_{s+1}^{\varepsilon}-\bar{X}_{s+1}-\xi_{s+1}\\
= & \left[  \widetilde{b}_{u}\left(  s\right)  -b_{u}\left(  s\right)
\right]  \varepsilon\Delta v+\sum\limits_{i=1}^{m}e_{i}\cdot\varepsilon\Delta
v^{\ast}\left[  \widetilde{\sigma}_{iu}\left(  s\right)  -\sigma_{iu}\left(
s\right)  \right]  M_{s+1}%
\end{array}
\]
where
\begin{align*}
\widetilde{b}_{u}\left(  s\right)   &  =\int_{0}^{1}b_{u}\left(  s,\bar{X}%
_{s},\bar{u}_{s}+\lambda\left(  u_{s}^{\varepsilon}-\bar{u}_{s}\right)
\right)  d\lambda,\\
\widetilde{\sigma}_{iu}\left(  s\right)   &  =\int_{0}^{1}\sigma_{iu}\left(
s,\bar{X}_{s},\bar{u}_{s}+\lambda\left(  u_{s}^{\varepsilon}-\bar{u}%
_{s}\right)  \right)  d\lambda.
\end{align*}
Then
\[%
\begin{array}
[c]{cl}
& \mathbb{E}\left\vert X_{s+1}^{\varepsilon}-\bar{X}_{s+1}-\xi_{s+1}%
\right\vert ^{2}\\
\leq & 2\mathbb{E}\left[  \left\vert \left[  \widetilde{b}_{u}\left(
s\right)  -b_{u}\left(  s\right)  \right]  \varepsilon\Delta v\right\vert
^{2}+\sum\limits_{i=1}^{m}\left(  \varepsilon\Delta v^{\ast}\left[
\widetilde{\sigma}_{iu}\left(  s\right)  -\sigma_{iu}\left(  s\right)
\right]  M_{s+1}\right)  ^{2}\right] \\
\leq & C\mathbb{E}\left[  \left\Vert \widetilde{b}_{u}\left(  s\right)
-b_{u}\left(  s\right)  \right\Vert ^{2}\left\vert \Delta v\right\vert
^{2}+\sum\limits_{i=1}^{m}\left\Vert \left[  \widetilde{\sigma}_{iu}\left(
s\right)  -\sigma_{iu}\left(  s\right)  \right]  \widetilde{I}\right\Vert
^{2}\left\vert \Delta v\right\vert ^{2}\right]  \varepsilon^{2}.
\end{array}
\]
Since $\left\Vert \widetilde{b}_{u}\left(  s\right)  -b_{u}\left(  s\right)
\right\Vert \rightarrow0$ and $\left\Vert \left[  \widetilde{\sigma}%
_{iu}\left(  s\right)  -\sigma_{iu}\left(  s\right)  \right]  \widetilde{I}%
\right\Vert \rightarrow0$ as $\varepsilon\rightarrow0$, we have%
\[
\lim_{\varepsilon\rightarrow0}\frac{1}{\varepsilon^{2}}\mathbb{E}\left\vert
X_{s+1}^{\varepsilon}-\bar{X}_{s+1}-\xi_{s+1}\right\vert ^{2}=0.
\]

When $t=s+2,...,T$,%
\[%
\begin{array}
[c]{cl}
& X_{t}^{\varepsilon}-\bar{X}_{t}-\xi_{t}\\
= & \widetilde{b}_{x}\left(  t-1\right)  \left(  X_{t-1}^{\varepsilon}-\bar
{X}_{t-1}-\xi_{t-1}\right)  +\left[  \widetilde{b}_{x}\left(  t-1\right)
-b_{x}\left(  t-1\right)  \right]  \xi_{t-1}\\
& +\sum\limits_{i=1}^{m}e_{i}\cdot\left\{  \left(  X_{t-1}^{\varepsilon}%
-\bar{X}_{t-1}-\xi_{t-1}\right)  ^{\ast}\widetilde{\sigma}_{ix}\left(
t-1\right)  +\xi_{t-1}^{\ast}\left[  \widetilde{\sigma}_{ix}\left(
t-1\right)  -\sigma_{ix}\left(  t-1\right)  \right]  \right\}  M_{t}%
\end{array}
\]
where
\begin{align*}
\widetilde{b}_{x}\left(  t\right)   &  =\int_{0}^{1}b_{x}\left(  t,\bar{X}%
_{t}+\lambda\left(  X_{t}^{\varepsilon}-\bar{X}_{t}\right)  ,\bar{u}%
_{t}\right)  d\lambda,\\
\widetilde{\sigma}_{ix}\left(  t\right)   &  =\int_{0}^{1}\sigma_{ix}\left(
t,\bar{X}_{t}+\lambda\left(  X_{t}^{\varepsilon}-\bar{X}_{t}\right)  ,\bar
{u}_{t}\right)  d\lambda.
\end{align*}
Then%
\[%
\begin{array}
[c]{cl}
& \mathbb{E}\left\vert X_{t}^{\varepsilon}-\bar{X}_{t}-\xi_{t}\right\vert
^{2}\\
\leq & C\mathbb{E}\left[  \left\Vert \widetilde{b}_{x}\left(  t-1\right)
\right\Vert ^{2}\left\vert X_{t-1}^{\varepsilon}-\bar{X}_{t-1}-\xi
_{t-1}\right\vert ^{2}+\left\Vert \widetilde{b}_{x}\left(  t-1\right)
-b_{x}\left(  t-1\right)  \right\Vert ^{2}\left\vert \xi_{t-1}\right\vert
^{2}\right. \\
& \left.  +\sum\limits_{i=1}^{m}\left\vert \widetilde{\sigma}_{ix}\left(
t-1\right)  M_{t}\right\vert ^{2}\left\vert X_{t-1}^{\varepsilon}-\bar
{X}_{t-1}-\xi_{t-1}\right\vert ^{2}+\sum\limits_{i=1}^{m}\left\vert \left[
\widetilde{\sigma}_{ix}\left(  t-1\right)  -\sigma_{ix}\left(  t-1\right)
\right]  M_{t}\right\vert ^{2}\left\vert \xi_{t-1}\right\vert ^{2}\right] \\
\leq & C\mathbb{E}\left[  \left(  \left\Vert \widetilde{b}_{x}\left(
t-1\right)  \right\Vert ^{2}+\sum\limits_{i=1}^{m}\left\Vert \widetilde{\sigma
}_{ix}\left(  t-1\right)  \widetilde{I}\right\Vert ^{2}\right)  \left\vert
X_{t-1}^{\varepsilon}-\bar{X}_{t-1}-\xi_{t-1}\right\vert ^{2}\right. \\
& +\left\Vert \widetilde{b}_{x}\left(  t-1\right)  -b_{x}\left(  t-1\right)
\right\Vert ^{2}\left\vert \xi_{t-1}\right\vert ^{2}\left.  +\sum
\limits_{i=1}^{m}\left\Vert \left[  \widetilde{\sigma}_{ix}\left(  t-1\right)
-\sigma_{ix}\left(  t-1\right)  \right]  \widetilde{I}\right\Vert
^{2}\left\vert \xi_{t-1}\right\vert ^{2}\right]  .
\end{array}
\]
It is easy to check that $\left\Vert \widetilde{b}_{x}\left(  t-1\right)
-b_{x}\left(  t-1\right)  \right\Vert \rightarrow0$ and $\left\Vert \left[
\widetilde{\sigma}_{ix}\left(  t-1\right)  -\sigma_{ix}\left(  t-1\right)
\right]  \widetilde{I}\right\Vert \rightarrow0$ as $\varepsilon\rightarrow0$.
Since $\widetilde{b}_{x}\left(  t-1\right)  $ and $\widetilde{\sigma}%
_{ix}\left(  t-1\right)  $ are bounded, by the estimation
(\ref{c5_eq_m_pc_xi_estimate}), we have%
\[
\lim_{\varepsilon\rightarrow0}\frac{1}{\varepsilon^{2}}\mathbb{E}\left\vert
X_{t}^{\varepsilon}-\bar{X}_{t}-\xi_{t}\right\vert ^{2}=0.
\]
This completes the proof.
\end{proof}

\begin{lemma}
Under Assumption \ref{c5_assumption_m_pc}, we have
\begin{align}
\sup_{0\leq t\leq T}\mathbb{E}\left\vert Y_{t}^{\varepsilon}-\bar{Y}%
_{t}\right\vert ^{2}  &  \leq C\varepsilon^{2}\mathbb{E}\left\vert \Delta
v\right\vert ^{2},\label{c5_eq_m_pc_y_estimate_1}\\
\sup_{0\leq t\leq T-1}\mathbb{E}\left\Vert \left(  Z_{t}^{\varepsilon}-\bar
{Z}_{t}\right)  \widetilde{I}\right\Vert ^{2}  &  \leq C\varepsilon
^{2}\mathbb{E}\left\vert \Delta v\right\vert ^{2}.
\label{c5_eq_m_pc_z_estimate_1}%
\end{align}

\end{lemma}

\begin{proof}
It is obvious that $Y_{T}^{\varepsilon}-\bar{Y}_{T}=0$ at time $T$.

When $t=s,...,T-1$ (if $s=T$, skip this part), we have%
\[%
\begin{array}
[c]{cl}
& \mathbb{E}\left\vert f\left(  t+1,X_{t+1}^{\varepsilon},Y_{t+1}%
^{\varepsilon},Z_{t+1}^{\varepsilon}\widetilde{I},\bar{u}_{t+1}\right)
-f\left(  t+1,\bar{X}_{t+1},\bar{Y}_{t+1},\bar{Z}_{t+1}\widetilde{I},\bar
{u}_{t+1}\right)  \right\vert ^{2}\\
\leq & C\mathbb{E}\left[  \left\vert X_{t+1}^{\varepsilon}-\bar{X}%
_{t+1}\right\vert ^{2}+\left\vert Y_{t+1}^{\varepsilon}-\bar{Y}_{t+1}%
\right\vert ^{2}+\left\Vert Z_{t+1}^{\varepsilon}\widetilde{I}-\bar{Z}%
_{t+1}\widetilde{I}\right\Vert ^{2}\right] \\
\leq & C\mathbb{E}\left[  \left\vert Y_{t+1}^{\varepsilon}-\bar{Y}%
_{t+1}\right\vert ^{2}+\left\Vert Z_{t+1}^{\varepsilon}\widetilde{I}-\bar
{Z}_{t+1}\widetilde{I}\right\Vert ^{2}\right]  +C_{1}\varepsilon^{2}%
\mathbb{E}\left\vert \Delta v\right\vert ^{2}.
\end{array}
\]
It yields that%
\[%
\begin{array}
[c]{cl}
& \mathbb{E}\left\vert Y_{t}^{\varepsilon}-\bar{Y}_{t}\right\vert ^{2}\\
\leq & C\mathbb{E}\left[  \left\vert Y_{t+1}^{\varepsilon}-\bar{Y}%
_{t+1}\right\vert ^{2}+\left\Vert Z_{t+1}^{\varepsilon}\widetilde{I}-\bar
{Z}_{t+1}\widetilde{I}\right\Vert ^{2}\right]  +C\varepsilon^{2}%
\mathbb{E}\left\vert \Delta v\right\vert ^{2}.
\end{array}
\]
Similarly, we have%
\[%
\begin{array}
[c]{cl}
& \mathbb{E}\left\vert \left(  Z_{t}^{\varepsilon}-\bar{Z}_{t}\right)
M_{t+1}\right\vert ^{2}\\
\leq & C\mathbb{E}\left[  \left\vert Y_{t+1}^{\varepsilon}-\bar{Y}%
_{t+1}\right\vert ^{2}+\left\Vert Z_{t+1}^{\varepsilon}\widetilde{I}-\bar
{Z}_{t+1}\widetilde{I}\right\Vert ^{2}\right]  +C\varepsilon^{2}%
\mathbb{E}\left\vert \Delta v\right\vert ^{2}.
\end{array}
\]
Combined with Proposition 2.4, we have%
\[%
\begin{array}
[c]{cl}
& \mathbb{E}\left[  \left\Vert \left(  Z_{t}^{\varepsilon}-\bar{Z}_{t}\right)
\widetilde{I}\right\Vert ^{2}\right] \\
\leq & C\mathbb{E}\left[  \left\vert \left(  Z_{t}^{\varepsilon}-\bar{Z}%
_{t}\right)  M_{t+1}\right\vert ^{2}\right] \\
\leq & C\mathbb{E}\left[  \left\vert Y_{t+1}^{\varepsilon}-\bar{Y}%
_{t+1}\right\vert ^{2}+\left\Vert Z_{t+1}^{\varepsilon}\widetilde{I}-\bar
{Z}_{t+1}\widetilde{I}\right\Vert ^{2}\right]  +C\varepsilon^{2}%
\mathbb{E}\left\vert \Delta v\right\vert ^{2}.
\end{array}
\]

When $t=s-1$, by similar analysis,%
\[%
\begin{array}
[c]{rcc}%
\mathbb{E}\left\vert Y_{t}^{\varepsilon}-\bar{Y}_{t}\right\vert ^{2} & \leq &
C\mathbb{E}\left[  \left\vert Y_{t+1}^{\varepsilon}-\bar{Y}_{t+1}\right\vert
^{2}+\left\Vert \left(  Z_{t+1}^{\varepsilon}-\bar{Z}_{t+1}\right)
\widetilde{I}\right\Vert ^{2}\right]  +C\varepsilon^{2}\mathbb{E}\left\vert
\Delta v\right\vert ^{2},\\
\mathbb{E}\left[  \left\Vert \left(  Z_{t}^{\varepsilon}-\bar{Z}_{t}\right)
\widetilde{I}\right\Vert ^{2}\right]  & \leq & C\mathbb{E}\left[  \left\vert
Y_{t+1}^{\varepsilon}-\bar{Y}_{t+1}\right\vert ^{2}+\left\Vert \left(
Z_{t+1}^{\varepsilon}-\bar{Z}_{t+1}\right)  \widetilde{I}\right\Vert
^{2}\right]  +C\varepsilon^{2}\mathbb{E}\left\vert \Delta v\right\vert ^{2}.
\end{array}
\]
If $s=T$,%
\[
\left\{
\begin{array}
[c]{rcc}%
\mathbb{E}\left\vert Y_{T-1}^{\varepsilon}-\bar{Y}_{T-1}\right\vert ^{2} &
\leq & C\varepsilon^{2}\mathbb{E}\left\vert \Delta v\right\vert ^{2},\\
\mathbb{E}\left[  \left\Vert \left(  Z_{T-1}^{\varepsilon}-\bar{Z}%
_{T-1}\right)  \widetilde{I}\right\Vert ^{2}\right]  & \leq & C\varepsilon
^{2}\mathbb{E}\left\vert \Delta v\right\vert ^{2}.
\end{array}
\right.
\]

When $t=0,...,s-2$, we have
\[%
\begin{array}
[c]{rcc}%
\mathbb{E}\left\vert Y_{t}^{\varepsilon}-\bar{Y}_{t}\right\vert ^{2} & \leq &
C\mathbb{E}\left[  \left\vert Y_{t+1}^{\varepsilon}-\bar{Y}_{t+1}\right\vert
^{2}+\left\Vert \left(  Z_{t+1}^{\varepsilon}-\bar{Z}_{t+1}\right)
\widetilde{I}\right\Vert ^{2}\right]  ,\\
\mathbb{E}\left[  \left\Vert \left(  Z_{t}^{\varepsilon}-\bar{Z}_{t}\right)
\widetilde{I}\right\Vert ^{2}\right]  & \leq & C\mathbb{E}\left[  \left\vert
Y_{t+1}^{\varepsilon}-\bar{Y}_{t+1}\right\vert ^{2}+\left\Vert \left(
Z_{t+1}^{\varepsilon}-\bar{Z}_{t+1}\right)  \widetilde{I}\right\Vert
^{2}\right]  .
\end{array}
\]

Thus, there exists $C>0$, such that for any $t\in\left\{  0,1,...,T\right\}
$,
\[
\left\{
\begin{array}
[c]{rcc}%
\mathbb{E}\left\vert Y_{t}^{\varepsilon}-\bar{Y}_{t}\right\vert ^{2} & \leq &
C\varepsilon^{2}\mathbb{E}\left\vert \Delta v\right\vert ^{2},\\
\mathbb{E}\left[  \left\Vert \left(  Z_{t}^{\varepsilon}-\bar{Z}_{t}\right)
\widetilde{I}\right\Vert ^{2}\right]  & \leq & C\varepsilon^{2}\mathbb{E}%
\left\vert \Delta v\right\vert ^{2}.
\end{array}
\right.
\]
This completes the proof.
\end{proof}

Let $\left(  \eta,\zeta\right)  $ be the solution to the following BS$\Delta$E,%

\[
\left\{
\begin{array}
[c]{rcl}%
\Delta\eta_{t} & = & -f_{x}\left(  t+1\right)  \xi_{t+1}-f_{y}\left(
t+1\right)  \eta_{t+1}-\delta_{\left(  t+1\right)  s}f_{u}\left(  t+1\right)
\varepsilon\Delta v\\
&  & -\sum\limits_{i=1}^{n}f_{\widetilde{z}_{i}}\left(  t+1\right)
\widetilde{I}^{\ast}\zeta_{t+1}^{\ast}e_{i}+\zeta_{t}M_{t+1},\\
\eta_{T} & = & 0.
\end{array}
\right.
\]

It is easy to check that%
\begin{align*}
\sup_{0\leq t\leq T}\mathbb{E}\left\vert \eta_{t}\right\vert ^{2}  &  \leq
C\varepsilon^{2}\mathbb{E}\left\vert \Delta v\right\vert ^{2},\\
\sup_{0\leq t\leq T-1}\mathbb{E}\left\Vert \zeta_{t}\widetilde{I}\right\Vert
^{2}  &  \leq C\varepsilon^{2}\mathbb{E}\left\vert \Delta v\right\vert ^{2}.
\end{align*}

and we have the following result:

\begin{lemma}
\label{c5_lemma_m_pc_yz_approxi}Under Assumption \ref{c5_assumption_m_pc}, we
have
\begin{align*}
\sup_{0\leq t\leq T}\mathbb{E}\left\vert Y_{t}^{\varepsilon}-\bar{Y}_{t}%
-\eta_{t}\right\vert ^{2}  &  =o\left(  \varepsilon^{2}\right)  ,\\
\sup_{0\leq t\leq T-1}\mathbb{E}\left\Vert \left(  Z_{t}^{\varepsilon}-\bar
{Z}_{t}-\zeta_{t}\right)  \widetilde{I}\right\Vert ^{2}  &  =o\left(
\varepsilon^{2}\right)  .
\end{align*}

\end{lemma}

\begin{proof}
When $t=T$, $Y_{T}^{\varepsilon}-\bar{Y}_{T}-\eta_{T}=0$.

When $t\in\left\{  0,1,...,T-1\right\}  $, we have%
\[%
\begin{array}
[c]{cl}
& Y_{t}^{\varepsilon}-\bar{Y}_{t}-\eta_{t}\\
= & \mathbb{E}\left[  Y_{t+1}^{\varepsilon}-\bar{Y}_{t+1}-\eta_{t+1}%
+f^{\varepsilon}\left(  t+1\right)  -\overline{f}\left(  t+1\right)  \right.
\\
& -f_{x}\left(  t+1\right)  \xi_{t+1}-f_{y}\left(  t+1\right)  \eta_{t+1}\\
& \left.  -\sum\limits_{i=1}^{n}f_{\widetilde{z}_{i}}\left(  t+1\right)
\widetilde{I}^{\ast}\zeta_{t+1}^{\ast}e_{i}-\delta_{\left(  t+1\right)
s}f_{u}\left(  t+1\right)  \varepsilon\Delta v|\mathcal{F}_{t}\right] \\
= & \mathbb{E}\left[  Y_{t+1}^{\varepsilon}-\bar{Y}_{t+1}-\eta_{t+1}\right. \\
& +\widetilde{f}_{x}\left(  t+1\right)  \left(  X_{t+1}^{\varepsilon}-\bar
{X}_{t+1}\right)  +\widetilde{f}_{y}\left(  t+1\right)  \left(  Y_{t+1}%
^{\varepsilon}-\bar{Y}_{t+1}\right) \\
& +\sum\limits_{i=1}^{n}\widetilde{f}_{\widetilde{z}_{i}}\left(  t+1\right)
\widetilde{I}^{\ast}\left(  Z_{t+1}^{\varepsilon}-\bar{Z}_{t+1}\right)
^{\ast}e_{i}+\delta_{\left(  t+1\right)  s}\widetilde{f}_{u}\left(
t+1\right)  \varepsilon\Delta v\\
& -f_{x}\left(  t+1\right)  \xi_{t+1}-f_{y}\left(  t+1\right)  \eta_{t+1}\\
& \left.  -\sum\limits_{i=1}^{n}f_{\widetilde{z}_{i}}\left(  t+1\right)
\widetilde{I}^{\ast}\zeta_{t+1}^{\ast}e_{i}-\delta_{\left(  t+1\right)
s}f_{u}\left(  t+1\right)  \varepsilon\Delta v|\mathcal{F}_{t}\right]  ,
\end{array}
\]
where%
\[
\widetilde{f}_{\mu}\left(  t\right)  =\int_{0}^{1}f_{\mu}(t,\bar{X}%
_{t}+\lambda\left(  X_{t}^{\varepsilon}-\bar{X}_{t}\right)  ,\bar{Y}%
_{t}+\lambda\left(  Y_{t}^{\varepsilon}-\bar{Y}_{t}\right)  ,\bar{Z}%
_{t}\widetilde{I}+\lambda\left(  Z_{t}^{\varepsilon}-\bar{Z}_{t}\right)
\widetilde{I},\bar{u}_{t}+\lambda\left(  u_{t}^{\varepsilon}-\bar{u}%
_{t}\right)  d\lambda
\]
for $\mu=x$, $y$, $z_{i}$ and $u$. Then,%
\[%
\begin{array}
[c]{cl}
& \mathbb{E}\left\vert Y_{t}^{\varepsilon}-\bar{Y}_{t}-\eta_{t}\right\vert
^{2}\\
\leq & C\mathbb{E}\left[  \left\vert Y_{t+1}^{\varepsilon}-\bar{Y}_{t+1}%
-\eta_{t+1}\right\vert ^{2}\right. \\
& +\left\vert \widetilde{f}_{x}\left(  t+1\right)  \left(  X_{t+1}%
^{\varepsilon}-\bar{X}_{t+1}-\xi_{t+1}\right)  \right\vert ^{2}+\left\vert
\left[  \widetilde{f}_{x}\left(  t+1\right)  -f_{x}\left(  t+1\right)
\right]  \xi_{t+1}\right\vert ^{2}\\
& +\left\vert \widetilde{f}_{y}\left(  t+1\right)  \left(  Y_{t+1}%
^{\varepsilon}-\bar{Y}_{t+1}-\eta_{t+1}\right)  \right\vert ^{2}+\left\vert
\left[  \widetilde{f}_{y}\left(  t+1\right)  -f_{y}\left(  t+1\right)
\right]  \eta_{t+1}\right\vert ^{2}\\
& +\sum\limits_{i=1}^{n}\left\vert \widetilde{f}_{\widetilde{z}_{i}}\left(
t+1\right)  \widetilde{I}^{\ast}\left(  Z_{t+1}^{\varepsilon}-\bar{Z}%
_{t+1}-\zeta_{t+1}\right)  ^{\ast}e_{i}\right\vert ^{2}+\sum\limits_{i=1}%
^{n}\left\vert \left[  \widetilde{f}_{\widetilde{z}_{i}}\left(  t+1\right)
-f_{\widetilde{z}_{i}}\left(  t+1\right)  \right]  \widetilde{I}^{\ast}%
\zeta_{t+1}^{\ast}e_{i}\right\vert ^{2}\\
& \left.  +\delta_{\left(  t+1\right)  s}\left\vert \left[  \widetilde{f}%
_{u}\left(  t+1\right)  -f_{u}\left(  t+1\right)  \right]  \varepsilon\Delta
v\right\vert ^{2}\right]
\end{array}
\]
and%
\[%
\begin{array}
[c]{cl}
& \mathbb{E}\left\Vert \left(  Z_{t}^{\varepsilon}-\bar{Z}_{t}-\zeta
_{t}\right)  \widetilde{I}\right\Vert ^{2}\\
\leq & C\mathbb{E}\left[  \left\vert Y_{t+1}^{\varepsilon}-\bar{Y}_{t+1}%
-\eta_{t+1}\right\vert ^{2}\right. \\
& +\left\vert \widetilde{f}_{x}\left(  t+1\right)  \left(  X_{t+1}%
^{\varepsilon}-\bar{X}_{t+1}-\xi_{t+1}\right)  \right\vert ^{2}+\left\vert
\left[  \widetilde{f}_{x}\left(  t+1\right)  -f_{x}\left(  t+1\right)
\right]  \xi_{t+1}\right\vert ^{2}\\
& +\left\vert \widetilde{f}_{y}\left(  t+1\right)  \left(  Y_{t+1}%
^{\varepsilon}-\bar{Y}_{t+1}-\eta_{t+1}\right)  \right\vert ^{2}+\left\vert
\left[  \widetilde{f}_{y}\left(  t+1\right)  -f_{y}\left(  t+1\right)
\right]  \eta_{t+1}\right\vert ^{2}\\
& +\sum\limits_{i=1}^{n}\left\vert \widetilde{f}_{\widetilde{z}_{i}}\left(
t+1\right)  \widetilde{I}^{\ast}\left(  Z_{t+1}^{\varepsilon}-\bar{Z}%
_{t+1}-\zeta_{t+1}\right)  ^{\ast}e_{i}\right\vert ^{2}+\sum\limits_{i=1}%
^{n}\left\vert \left[  \widetilde{f}_{\widetilde{z}_{i}}\left(  t+1\right)
-f_{\widetilde{z}_{i}}\left(  t+1\right)  \right]  \widetilde{I}^{\ast}%
\zeta_{t+1}^{\ast}e_{i}\right\vert ^{2}\\
& \left.  +\delta_{\left(  t+1\right)  s}\left\vert \left[  \widetilde{f}%
_{u}\left(  t+1\right)  -f_{u}\left(  t+1\right)  \right]  \varepsilon\Delta
v\right\vert ^{2}\right]  .
\end{array}
\]
Notice that $\widetilde{f}_{x}\left(  t\right)  -f_{x}\left(  t\right)
\rightarrow0,$ $\widetilde{f}_{y}\left(  t\right)  -f_{y}\left(  t\right)
\rightarrow0,$ $\widetilde{f}_{z_{i}}\left(  t\right)  -f_{z_{i}}\left(
t\right)  \rightarrow0,$ $\widetilde{f}_{u}\left(  t\right)  -f_{u}\left(
t\right)  \rightarrow0$ as $\varepsilon\rightarrow0$. We obtain that%
\[%
\begin{array}
[c]{rcc}%
\lim_{\varepsilon\rightarrow0}\frac{1}{\varepsilon^{2}}\mathbb{E}\left\vert
Y_{t}^{\varepsilon}-\bar{Y}_{t}-\eta_{t}\right\vert ^{2} & = & 0,\\
\lim_{\varepsilon\rightarrow0}\frac{1}{\varepsilon^{2}}\mathbb{E}\left\Vert
\left(  Z_{t}^{\varepsilon}-\bar{Z}_{t}-\zeta_{t}\right)  \widetilde{I}%
\right\Vert ^{2} & = & 0.
\end{array}
\]
This completes the proof.
\end{proof}

By Lemma \ref{c5_lemma_m_pc_x_approxi} and Lemma
\ref{c5_lemma_m_pc_yz_approxi}, we have%
\begin{align*}
&
\begin{array}
[c]{cc}
& J\left(  u^{\varepsilon}\left(  \cdot\right)  \right)  -J\left(  \bar
{u}\left(  \cdot\right)  \right)
\end{array}
\\
&
\begin{array}
[c]{cc}%
= & \mathbb{E}\sum_{t=0}^{T-1}\left[  \left\langle l_{x}\left(  t\right)
,\xi_{t}\right\rangle +\left\langle l_{y}\left(  t\right)  ,\eta
_{t}\right\rangle +\sum\limits_{i=1}^{n}\left\langle l_{\widetilde{z}_{i}%
}^{\ast}\left(  t\right)  ,\widetilde{I}\zeta_{t}^{\ast}e_{i}\right\rangle
+\delta_{ts}\left\langle l_{u}\left(  s\right)  ,\varepsilon\Delta
v\right\rangle \right] \\
& +\mathbb{E}\left\langle h_{x}\left(  \bar{X}_{T}\right)  ,\xi_{T}%
\right\rangle +o\left(  \varepsilon\right)  .
\end{array}
\end{align*}

Introducing the following adjoint equation:%

\begin{equation}
\left\{
\begin{array}
[c]{rcl}%
\Delta p_{t} & = & -b_{x}^{\ast}\left(  t+1\right)  p_{t+1}-\sum
\limits_{i=1}^{m}\sigma_{ix}\left(  t+1\right)  \mathbb{E}\left[
M_{t+2}M_{t+2}^{\ast}|\mathcal{F}_{t+1}\right]  q_{t+1}^{\ast}e_{i}\\
&  & +f_{x}^{\ast}\left(  t+1\right)  k_{t+1}+l_{x}\left(  t+1\right)
+q_{t}M_{t+1},\\
\Delta k_{t} & = & f_{y}^{\ast}\left(  t\right)  k_{t}+l_{y}\left(  t\right)
+\left[  \sum\limits_{i=1}^{n}e_{i}k_{t}^{\ast}f_{\widetilde{z}_{i}}\left(
t\right)  \widetilde{I}^{\ast}\left(  \mathbb{E}\left[  M_{t+1}M_{t+1}^{\ast
}|\mathcal{F}_{t}\right]  \right)  ^{\dag}\right. \\
& \multicolumn{1}{r}{} & \multicolumn{1}{r}{\left.  +\sum\limits_{i=1}%
^{n}e_{i}l_{\widetilde{z}_{i}}\left(  t\right)  \widetilde{I}^{\ast}\left(
\mathbb{E}\left[  M_{t+1}M_{t+1}^{\ast}|\mathcal{F}_{t}\right]  \right)
^{\dag}\right]  M_{t+1},}\\
p_{T} & = & -h_{x}\left(  \bar{X}_{T}\right)  ,\\
k_{0} & = & 0,
\end{array}
\right.  \label{c5_eq_m_pc_adjoint_eq}%
\end{equation}
where $\left(  \cdot\right)  ^{\dag}$ denotes the pseudoinverse of a matrix.

Obviously the forward equation in (\ref{c5_eq_m_pc_adjoint_eq}) admits a
unique solution $k\in\mathcal{M}\left(  0,T;\mathbb{R}^{n}\right)  $. Then,
based on the solution $k$, according to Theorem \ref{bsde_result_l}, it is
easy to check that the backward equation in (\ref{c5_eq_m_pc_adjoint_eq}) has
a unique solution $\left(  p,q\right)  \in\mathcal{M}\left(  0,T;\mathbb{R}%
^{m}\right)  \times\mathcal{M}\left(  0,T-1;\mathbb{R}^{m\times d}\right)  $.
So FBS$\Delta$E has a unique solution $\left(  p,q,k\right)  $.

We obtain the following maximum principle for the optimal control problem
(\ref{c5_eq_m_pc_state_eq})-(\ref{c5_eq_m_pc_cost_eq}).

Define the Hamiltonian function%
\[%
\begin{array}
[c]{l}%
H\left(  \omega,t,u,x,y,z,p,q,k\right) \\%
\begin{array}
[c]{cl}%
= & b^{\ast}\left(  \omega,t,x,u\right)  p+\sum\limits_{i=1}^{m}\sigma
_{i}\left(  \omega,t,x,u\right)  \mathbb{E}\left[  M_{t+1}M_{t+1}^{\ast
}|\mathcal{F}_{t}\right]  \left(  \omega\right)  q^{\ast}e_{i}\\
& -f^{\ast}\left(  \omega,t,x,y,\widetilde{z},u\right)  k-l\left(
\omega,t,x,y,\widetilde{z},u\right)  .
\end{array}
\end{array}
\]

\begin{theorem}
Suppose that Assumption \ref{c5_assumption_m_pc} holds. Let $\bar{u}$ be an
optimal control of the problem (\ref{c5_eq_m_pc_state_eq}%
)-(\ref{c5_eq_m_pc_cost_eq}), $\left(  \bar{X},\bar{Y},\bar{Z}\right)  $ be
the corresponding optimal trajectory and $\left(  p,q,k\right)  $ be the
solution to the adjoint equation (\ref{c5_eq_m_pc_adjoint_eq}). Then for any
$t\in\left\{  0,1,...,T\right\}  $, $v\in U_{t}$ and $\omega\in\Omega$, we
have%
\begin{equation}
\left\langle H_{u}\left(  \omega,t,\bar{u}_{t},\bar{X}_{t},\bar{Y}_{t},\bar
{Z}_{t}\widetilde{I},p_{t},q_{t},k_{t}\right)  ,v-\bar{u}_{t}\left(
\omega\right)  \right\rangle \leq0. \label{c5_eq_m_pc_MP_result}%
\end{equation}

\end{theorem}

\begin{proof}
For $t\in\left\{  0,1,...,T-1\right\}  $, we have%
\begin{equation}%
\begin{array}
[c]{rl}
& \Delta\left\langle \xi_{t},p_{t}\right\rangle \\
= & \left\langle \xi_{t+1},\Delta p_{t}\right\rangle +\left\langle \Delta
\xi_{t},p_{t}\right\rangle \\
= & \left\langle \xi_{t+1},-b_{x}^{\ast}\left(  t+1\right)  p_{t+1}%
-\sum\nolimits_{i=1}^{m}\sigma_{ix}\left(  t+1\right)  \mathbb{E}\left[
M_{t+2}M_{t+2}^{\ast}|\mathcal{F}_{t+1}\right]  q_{t+1}^{\ast}e_{i}%
\right\rangle \\
& +\left\langle \xi_{t+1},f_{x}^{\ast}\left(  t+1\right)  k_{t+1}+l_{x}\left(
t+1\right)  \right\rangle \\
& +\left\langle \sum_{i=1}^{m}e_{i}\cdot\left[  \xi_{t}^{\ast}\sigma
_{ix}\left(  t\right)  +\delta_{ts}\varepsilon\Delta v^{\ast}\sigma
_{iu}\left(  t\right)  \right]  M_{t+1},q_{t}M_{t+1}\right\rangle
+\left\langle b_{x}\left(  t\right)  \xi_{t}+\delta_{ts}b_{u}\left(  t\right)
\varepsilon\Delta v,p_{t}\right\rangle +\Phi_{t},
\end{array}
\label{product-rule-1}%
\end{equation}
where%
\[%
\begin{array}
[c]{cl}%
\Phi_{t}= & \left\langle \xi_{t}+b_{x}\left(  t\right)  \xi_{t}+\delta
_{ts}b_{u}\left(  t\right)  \varepsilon\Delta v,q_{t}M_{t+1}\right\rangle \\
& +\left\langle \sum_{i=1}^{m}e_{i}\cdot\left[  \xi_{t}^{\ast}\sigma
_{ix}\left(  t\right)  +\delta_{ts}\varepsilon\Delta v^{\ast}\sigma
_{iu}\left(  t\right)  \right]  M_{t+1},p_{t}\right\rangle .
\end{array}
\]
It is obvious that $\mathbb{E}\left[  \Phi_{t}\right]  =0$. We have%
\[%
\begin{array}
[c]{cl}
& \mathbb{E}\left\langle \sum_{i=1}^{m}e_{i}\cdot\xi_{t}^{\ast}\sigma
_{ix}\left(  t\right)  M_{t+1},q_{t}M_{t+1}\right\rangle \\
= & \mathbb{E}\sum_{i=1}^{m}\left[  M_{t+1}^{\ast}q_{t}^{\ast}e_{i}\xi
_{t}^{\ast}\sigma_{ix}\left(  t\right)  M_{t+1}\right] \\
= & \mathbb{E}\sum_{i=1}^{m}\left[  e_{i}^{\ast}q_{t}M_{t+1}M_{t+1}^{\ast
}\sigma_{ix}^{\ast}\left(  t\right)  \xi_{t}\right] \\
= & \mathbb{E}\sum_{i=1}^{m}\left\langle \xi_{t},\sigma_{ix}\left(  t\right)
M_{t+1}M_{t+1}^{\ast}q_{t}^{\ast}e_{i}\right\rangle ,
\end{array}
\]
and%
\[
\mathbb{E}\left\langle \sum_{i=1}^{m}e_{i}\cdot\delta_{ts}\varepsilon\Delta
v^{\ast}\sigma_{iu}\left(  t\right)  M_{t+1},q_{t}M_{t+1}\right\rangle
=\mathbb{E}\left[  \delta_{ts}\varepsilon\sum_{i=1}^{m}\left\langle \Delta
v,\sigma_{iu}\left(  t\right)  M_{t+1}M_{t+1}^{\ast}q_{t}^{\ast}%
e_{i}\right\rangle \right]  .
\]

Similarly, it can be shown that for $t\in\left\{  0,1,...,T-1\right\}  $,%
\[%
\begin{array}
[c]{rl}
& \Delta\left\langle \eta_{t},k_{t}\right\rangle \\
= & \left\langle -f_{x}\left(  t+1\right)  \xi_{t+1}-f_{y}\left(  t+1\right)
\eta_{t+1}-\sum_{i=1}^{n}f_{\widetilde{z}_{i}}\left(  t+1\right)
\widetilde{I}^{\ast}\zeta_{t+1}^{\ast}e_{i},k_{t+1}\right\rangle \\
& -\left\langle \delta_{\left(  t+1\right)  s}f_{u}\left(  t+1\right)
\varepsilon\Delta v,k_{t+1}\right\rangle +\left\langle \eta_{t},f_{y}^{\ast
}\left(  t\right)  k_{t}+l_{y}\left(  t\right)  \right\rangle \\
& +\left\langle \zeta_{t}M_{t+1},\sum_{i=1}^{n}e_{i}\left(  k_{t}^{\ast
}f_{\widetilde{z}_{i}}\left(  t\right)  +l_{\widetilde{z}_{i}}\left(
t\right)  \right)  \widetilde{I}^{\ast}\left(  \mathbb{E}\left[
M_{t+1}M_{t+1}^{\ast}|\mathcal{F}_{t}\right]  \right)  ^{\dag}M_{t+1}%
\right\rangle +\Psi_{t},
\end{array}
\]
where%
\[%
\begin{array}
[c]{cl}%
\Psi_{t}= & \left\langle \zeta_{t}M_{t+1},k_{t}+g_{y}^{\ast}\left(  t\right)
k_{t}+l_{y}\left(  t\right)  \right\rangle +\left\langle \eta_{t}%
,\sum\limits_{i=1}^{n}e_{i}k_{t}^{\ast}f_{\widetilde{z}_{i}}\left(  t\right)
\widetilde{I}^{\ast}\left(  \mathbb{E}\left[  M_{t+1}M_{t+1}^{\ast
}|\mathcal{F}_{t}\right]  \right)  ^{\dag}M_{t+1}\right\rangle \\
& +\left\langle \eta_{t},\sum\limits_{i=1}^{n}e_{i}l_{\widetilde{z}_{i}%
}\left(  t\right)  \widetilde{I}^{\ast}\left(  \mathbb{E}\left[
M_{t+1}M_{t+1}^{\ast}|\mathcal{F}_{t}\right]  \right)  ^{\dag}M_{t+1}%
\right\rangle .
\end{array}
\]
According to the result in \cite{ce08}, we know that $\forall\omega\in\Omega$,%
\[%
\begin{array}
[c]{cl}
& \mathbb{E}\left[  M_{t+1}M_{t+1}^{\ast}|\mathcal{F}_{t}\right]  \left(
\mathbb{E}\left[  M_{t+1}M_{t+1}^{\ast}|\mathcal{F}_{t}\right]  \right)
^{\dag}\left(  \omega\right) \\
= & \left(  \mathbb{E}\left[  M_{t+1}M_{t+1}^{\ast}|\mathcal{F}_{t}\right]
\right)  ^{\dag}\mathbb{E}\left[  M_{t+1}M_{t+1}^{\ast}|\mathcal{F}%
_{t}\right]  \left(  \omega\right) \\
= & I_{d}-\frac{1}{d}\mathbf{1}_{d\times d}.
\end{array}
\]
Then we can obtain%
\[%
\begin{array}
[c]{cl}
& \mathbb{E}\left\langle \zeta_{t}M_{t+1},e_{i}k_{t}^{\ast}f_{\widetilde{z}%
_{i}}\left(  t\right)  \widetilde{I}^{\ast}\left(  \mathbb{E}\left[
M_{t+1}M_{t+1}^{\ast}|\mathcal{F}_{t}\right]  \right)  ^{\dag}M_{t+1}%
\right\rangle \\
= & \mathbb{E}\left[  M_{t+1}^{\ast}\zeta_{t}^{\ast}e_{i}k_{t}^{\ast
}f_{\widetilde{z}_{i}}\left(  t\right)  \widetilde{I}^{\ast}\left(
\mathbb{E}\left[  M_{t+1}M_{t+1}^{\ast}|\mathcal{F}_{t}\right]  \right)
^{\dag}M_{t+1}\right] \\
= & \mathbb{E}\left[  e_{i}^{\ast}\zeta_{t}M_{t+1}M_{t+1}^{\ast}\left(
\mathbb{E}\left[  M_{t+1}M_{t+1}^{\ast}|\mathcal{F}_{t}\right]  \right)
^{\dag}\widetilde{I}f_{\widetilde{z}_{i}}^{\ast}\left(  t\right)  k_{t}\right]
\\
= & \mathbb{E}\left[  e_{i}^{\ast}\zeta_{t}\left(  I_{d}-\frac{1}{d}%
\mathbf{1}_{d\times d}\right)  \widetilde{I}f_{\widetilde{z}_{i}}^{\ast
}\left(  t\right)  k_{t}\right] \\
= & \mathbb{E}\left[  e_{i}^{\ast}\zeta_{t}\widetilde{I}f_{\widetilde{z}_{i}%
}^{\ast}\left(  t\right)  k_{t}\right] \\
= & \mathbb{E}\left\langle f_{\widetilde{z}_{i}}\left(  t\right)
\widetilde{I}^{\ast}\zeta_{t}^{\ast}e_{i},k_{t}\right\rangle .
\end{array}
\]
Similarly,%
\[%
\begin{array}
[c]{cl}
& \mathbb{E}\left\langle \zeta_{t}M_{t+1},e_{i}l_{\widetilde{z}_{i}}\left(
t\right)  \widetilde{I}\left(  \mathbb{E}\left[  M_{t+1}M_{t+1}^{\ast
}|\mathcal{F}_{t}\right]  \right)  ^{\dag}M_{t+1}\right\rangle =\mathbb{E}%
\left[  l_{\widetilde{z}_{i}}\left(  t\right)  \widetilde{I}^{\ast}\zeta
_{t}^{\ast}e_{i}\right]  .
\end{array}
\]
Thus%
\begin{equation}%
\begin{array}
[c]{cl}
& \mathbb{E}\Delta\left(  \left\langle \xi_{t},p_{t}\right\rangle
+\left\langle \eta_{t},k_{t}\right\rangle \right) \\
= & \mathbb{E}\left[  \left\langle -b_{x}\left(  t+1\right)  \xi_{t+1}%
,p_{t+1}\right\rangle +\left\langle b_{x}\left(  t\right)  \xi_{t}%
,p_{t}\right\rangle \right. \\
& -\sum\limits_{i=1}^{m}\left\langle e_{i}M_{t+2}^{\ast}\sigma_{ix}^{\ast
}\left(  t+1\right)  \xi_{t+1},q_{t+1}M_{t+2}\right\rangle +\sum
\limits_{i=1}^{m}\left\langle e_{i}\xi_{t}^{\ast}\sigma_{ix}\left(  t\right)
M_{t+1},q_{t}M_{t+1}\right\rangle \\
& -\left\langle f_{y}\left(  t+1\right)  \eta_{t+1},k_{t+1}\right\rangle
+\left\langle f_{y}\left(  t\right)  \eta_{t},k_{t}\right\rangle \\
& -\sum\limits_{i=1}^{n}\left\langle f_{\widetilde{z}_{i}}\left(  t+1\right)
\widetilde{I}^{\ast}\zeta_{t+1}^{\ast}e_{i},k_{t+1}\right\rangle
+\sum\limits_{i=1}^{n}\left\langle f_{\widetilde{z}_{i}}\left(  t\right)
\widetilde{I}^{\ast}\zeta_{t}^{\ast}e_{i},k_{t}\right\rangle \\
& +\left\langle l_{x}\left(  t+1\right)  ,\xi_{t+1}\right\rangle +\left\langle
\eta_{t},l_{y}\left(  t\right)  \right\rangle +\sum\limits_{i=1}%
^{n}\left\langle l_{\widetilde{z}_{i}}^{\ast}\left(  t\right)  ,\widetilde{I}%
^{\ast}\zeta_{t}^{\ast}e_{i}\right\rangle \\
& +\varepsilon\left\langle \delta_{ts}b_{u}\left(  t\right)  \Delta
v,p_{t}\right\rangle +\delta_{ts}\varepsilon\sum\limits_{i=1}^{m}\left\langle
M_{t+1}M_{t+1}^{\ast}\sigma_{iu}^{\ast}\left(  t\right)  \Delta v,q_{t}^{\ast
}e_{i}\right\rangle \\
& \left.  -\varepsilon\left\langle \delta_{\left(  t+1\right)  s}f_{u}\left(
t+1\right)  \Delta v,k_{t+1}\right\rangle \right]  .
\end{array}
\label{product-rule-2}%
\end{equation}
Therefore,%
\begin{equation}%
\begin{array}
[c]{cl}
& -\mathbb{E}\left\langle h_{x}\left(  \bar{X}_{T}\right)  ,\xi_{T}%
\right\rangle \\
= & \mathbb{E}\left[  \left\langle \xi_{T},p_{T}\right\rangle +\left\langle
\eta_{T},k_{T}\right\rangle -\left\langle \xi_{0},p_{0}\right\rangle
-\left\langle \eta_{0},k_{0}\right\rangle \right] \\
= & \sum\limits_{t=0}^{T-1}\mathbb{E}\Delta\left(  \left\langle \xi_{t}%
,p_{t}\right\rangle +\left\langle \eta_{t},k_{t}\right\rangle \right) \\
= & \mathbb{E}\left[  \left\langle b_{x}\left(  0\right)  \xi_{0}%
,p_{0}\right\rangle +\sum\limits_{i=1}^{m}\left\langle e_{i}\xi_{0}^{\ast
}\sigma_{ix}\left(  0\right)  M_{1},q_{0}M_{1}\right\rangle +\left\langle
f_{y}\left(  0\right)  \eta_{0},k_{0}\right\rangle +\sum\limits_{i=1}%
^{n}\left\langle f_{\widetilde{z}_{i}}\left(  0\right)  \widetilde{I}^{\ast
}\zeta_{0}^{\ast}e_{i},k_{0}\right\rangle \right] \\
& +\sum\limits_{t=0}^{T-1}\mathbb{E}\left[  \left\langle l_{x}\left(
t\right)  ,\xi_{t}\right\rangle +\left\langle l_{y}\left(  t\right)  ,\eta
_{t}\right\rangle +\sum\limits_{i=1}^{n}\left\langle l_{\widetilde{z}_{i}%
}^{\ast}\left(  t\right)  ,\widetilde{I}^{\ast}\zeta_{t}^{\ast}e_{i}%
\right\rangle \right] \\
& +\sum\limits_{t=0}^{T}\delta_{ts}\varepsilon\mathbb{E}\left[  \left\langle
b_{u}^{\ast}\left(  t\right)  p_{t},\Delta v\right\rangle +\sum\limits_{i=1}%
^{m}\left\langle \sigma_{iu}\left(  t\right)  M_{t+1}M_{t+1}^{\ast}q_{t}%
^{\ast}e_{i},\Delta v\right\rangle -\left\langle g_{u}^{\ast}\left(  t\right)
k_{t},\Delta v\right\rangle \right]  .
\end{array}
\label{rearrangement}%
\end{equation}
Since $\xi_{0}=0$ and $k_{0}=0$, we deduce%
\begin{equation}%
\begin{array}
[c]{cl}
& \mathbb{E}\sum\limits_{t=0}^{T-1}\left[  \left\langle l_{x}\left(  t\right)
,\xi_{t}\right\rangle +\left\langle l_{y}\left(  t\right)  ,\eta
_{t}\right\rangle +\sum\limits_{i=1}^{n}\left\langle l_{\widetilde{z}_{i}%
}^{\ast}\left(  t\right)  ,\widetilde{I}^{\ast}\zeta_{t}^{\ast}e_{i}%
\right\rangle \right]  +\mathbb{E}\left\langle h_{x}\left(  \bar{X}%
_{T}\right)  ,\xi_{T}\right\rangle \\
= & -\varepsilon\mathbb{E}\left[  \left\langle b_{u}^{\ast}\left(  s\right)
p_{s}+\sum\limits_{i=1}^{m}\sigma_{iu}\left(  s\right)  M_{s+1}M_{s+1}^{\ast
}q_{s}^{\ast}e_{i}-g_{u}^{\ast}\left(  s\right)  k_{s},\Delta v\right\rangle
\right]  .
\end{array}
\label{dual-relation}%
\end{equation}
By $\lim_{\varepsilon\rightarrow0}\frac{1}{\varepsilon}\left[  J\left(
u^{\varepsilon}\left(  \cdot\right)  \right)  -J\left(  \bar{u}\left(
\cdot\right)  \right)  \right]  \geq0$, we obtain%
\[
\mathbb{E}\left[  \left\langle b_{u}^{\ast}\left(  s\right)  p_{s}+\sum
_{i=1}^{m}\sigma_{iu}\left(  s\right)  M_{s+1}M_{s+1}^{\ast}q_{s}^{\ast}%
e_{i}-g_{u}^{\ast}\left(  s\right)  k_{s}-l_{u}\left(  s\right)  ,\Delta
v\right\rangle \right]  \leq0.
\]
It is easy to obtain equation (\ref{c5_eq_m_pc_MP_result}) since $s$ is taking
arbitrarily. This completes the proof.
\end{proof}

\begin{remark}
\label{re-product-rule}In the introduction we point out that we need a
reasonable representation of the product rule. When we calculate
$\Delta\left\langle \xi_{t},p_{t}\right\rangle $ in (\ref{product-rule-1}),
$\Delta\left\langle \xi_{t},p_{t}\right\rangle $ is represented as
$\left\langle \xi_{t+1},\cdot\cdot\cdot\right\rangle +\cdot\cdot\cdot$.
Combining the formulation of the BS$\Delta$E mentioned in the introduction,
this representation will lead to the terms such as $\left\langle \square
_{t},\Diamond_{t}\right\rangle -\left\langle \square_{t+1},\Diamond
_{t+1}\right\rangle $ in (\ref{product-rule-2}). By summing and rearranging
these terms in (\ref{rearrangement}), we obtain the dual relation
(\ref{dual-relation}).
\end{remark}

\section{Maximum principle for the fully coupled FBS$\Delta$E system}

In this section we consider the control problem (\ref{f_c_c_2_state_eq}%
)-(\ref{f_c_c_2_cost_eq}). Without loss of generality, we only consider the
one-dimensional case for $X$ and $Y$. Let $\bar{u}=\left\{  \bar{u}%
_{t}\right\}  _{t=0}^{T}$ be the optimal control for the control problem
(\ref{f_c_c_2_state_eq})-(\ref{f_c_c_2_cost_eq}) and $\left(  \bar{X},\bar
{Y},\bar{Z}\right)  $ be the corresponding optimal trajectory. Note that the
existence and uniqueness of $\left(  \bar{X},\bar{Y},\bar{Z}\right)  $ is
guaranteed by the results in \cite{ji-liu}. The perturbed control
$u^{\varepsilon}$ is the same as (\ref{perturb-control}) and we denote by
$\left(  X^{\varepsilon},Y^{\varepsilon},Z^{\varepsilon}\right)  $\thinspace
the corresponding trajectory.

Let%
\[
\widehat{X}_{t}=X_{t}^{\varepsilon}-\bar{X}_{t},\text{ }\widehat{Y}_{t}%
=Y_{t}^{\varepsilon}-\bar{Y}_{t},\text{ }\widehat{Z}_{t}=Z_{t}^{\varepsilon
}-\bar{Z}_{t}.
\]

Using the similar analysis and similar notations in section 3, we have%
\begin{equation}
\left\{
\begin{array}
[c]{rcl}%
\Delta\widehat{X}_{t} & = & b^{\varepsilon}\left(  t\right)  -\overline
{b}\left(  t\right)  +\left(  \sigma^{\varepsilon}\left(  t\right)
-\overline{\sigma}\left(  t\right)  \right)  M_{t+1},\\
\Delta\widehat{Y}_{t} & = & -f^{\varepsilon}\left(  t+1\right)  +\overline
{f}\left(  t+1\right)  +\widehat{Z}_{t}M_{t+1},\\
\widehat{X}_{0} & = & 0,\\
\widehat{Y}_{T} & = & 0.
\end{array}
\right.  \label{f_c_c_2_difference_1}%
\end{equation}

\begin{lemma}
\label{f_c_c_2_xyz_approxi}Under Assumption \ref{c5_assumption_m_pc} and
Assumption \ref{f_c_c_2_control_assumption}, we have%
\begin{equation}
\mathbb{E}\left(  \sum_{t=0}^{T}\left\vert \widehat{X}_{t}\right\vert
^{2}+\sum_{t=0}^{T}\left\vert \widehat{Y}_{t}\right\vert ^{2}+\sum_{t=0}%
^{T-1}\left\vert \widehat{Z}_{t}\widetilde{I}\right\vert ^{2}\right)  \leq
C\varepsilon^{2}\mathbb{E}\left\vert \Delta v\right\vert ^{2}.
\label{f_c_c_2_x_estimate_1}%
\end{equation}

\end{lemma}

\begin{proof}
By (\ref{f_c_c_2_difference_1}),%
\[%
\begin{array}
[c]{cl}
& \mathbb{E}\sum\limits_{t=0}^{T-1}\Delta\left\langle \widehat{X}%
_{t},\widehat{Y}_{t}\right\rangle =\mathbb{E}\left\langle \widehat{X}%
_{T},\widehat{Y}_{T}\right\rangle -\mathbb{E}\left\langle \widehat{X}%
_{0},\widehat{Y}_{0}\right\rangle =0\\
= & \mathbb{E}\sum\limits_{t=0}^{T}\left[  \left\langle \widehat{X}%
_{t},-f^{\varepsilon}\left(  t\right)  +\overline{f}\left(  t\right)
\right\rangle +\left\langle \widehat{Y}_{t},b^{\varepsilon}\left(  t\right)
-\overline{b}\left(  t\right)  \right\rangle +\left\langle \widehat{Z}%
_{t},\sigma^{\varepsilon}\left(  t\right)  -\overline{\sigma}\left(  t\right)
\right\rangle \right] \\
= & \mathbb{E}\sum\limits_{t=1}^{T-1}\left\langle A\left(  t,\lambda
_{t}^{\varepsilon};u_{t}^{\varepsilon}\right)  -A\left(  t,\overline{\lambda
}_{t};u_{t}^{\varepsilon}\right)  ,\widehat{\lambda}_{t}\right\rangle
+\mathbb{E}\left\langle \widehat{X}_{T},-f^{\varepsilon}\left(  T\right)
+\widetilde{f}^{\varepsilon}\left(  T\right)  \right\rangle \\
& +\mathbb{E}\left\langle \widehat{Y}_{0},b^{\varepsilon}\left(  0\right)
-\widetilde{b}^{\varepsilon}\left(  0\right)  \right\rangle +\mathbb{E}%
\left\langle \widehat{Z}_{0},\left(  \sigma^{\varepsilon}\left(  0\right)
-\widetilde{\sigma}^{\varepsilon}\left(  0\right)  \right)  \mathbb{E}\left[
M_{1}M_{1}^{\ast}|\mathcal{F}_{0}\right]  \right\rangle \\
& +\mathbb{E}\sum\limits_{t=0}^{T}\left[  \left\langle \widehat{X}%
_{t},-\widetilde{f}^{\varepsilon}\left(  t\right)  +\overline{f}\left(
t\right)  \right\rangle +\left\langle \widehat{Y}_{t},\widetilde{b}%
^{\varepsilon}\left(  t\right)  -\overline{b}\left(  t\right)  \right\rangle
\right] \\
& +\mathbb{E}\sum\limits_{t=0}^{T}\left\langle \widehat{Z}_{t},\left(
\widetilde{\sigma}^{\varepsilon}\left(  t\right)  -\overline{\sigma}\left(
t\right)  \right)  \mathbb{E}\left[  M_{t+1}M_{t+1}^{\ast}|\mathcal{F}%
_{t}\right]  \right\rangle \\
= & \mathbb{E}\sum\limits_{t=1}^{T-1}\left\langle A\left(  t,\lambda
_{t}^{\varepsilon};u_{t}^{\varepsilon}\right)  -A\left(  t,\overline{\lambda
}_{t};u_{t}^{\varepsilon}\right)  ,\widehat{\lambda}_{t}\right\rangle
+\mathbb{E}\left\langle \widehat{X}_{T},-f^{\varepsilon}\left(  T\right)
+\widetilde{f}^{\varepsilon}\left(  T\right)  \right\rangle \\
& +\mathbb{E}\left\langle \widehat{Y}_{0},b^{\varepsilon}\left(  0\right)
-\widetilde{b}^{\varepsilon}\left(  0\right)  \right\rangle +\mathbb{E}%
\left\langle \widehat{Z}_{0},\left(  \sigma^{\varepsilon}\left(  0\right)
-\widetilde{\sigma}^{\varepsilon}\left(  0\right)  \right)  \mathbb{E}\left[
M_{1}M_{1}^{\ast}|\mathcal{F}_{0}\right]  \right\rangle \\
& +\mathbb{E}\left[  \left\langle \widehat{X}_{s},-\widetilde{f}^{\varepsilon
}\left(  s\right)  +\overline{f}\left(  s\right)  \right\rangle +\left\langle
\widehat{Y}_{s},\widetilde{b}^{\varepsilon}\left(  s\right)  -\overline
{b}\left(  s\right)  \right\rangle \right] \\
& +\mathbb{E}\left\langle \widehat{Z}_{s}M_{s+1},\left(  \widetilde{\sigma
}^{\varepsilon}\left(  s\right)  -\overline{\sigma}\left(  s\right)  \right)
M_{s+1}\right\rangle .
\end{array}
\]
By the monotone condition, we obtain%
\begin{equation}%
\begin{array}
[c]{cl}
& \mathbb{E}\left[  \left\langle \widehat{X}_{s},-\widetilde{f}^{\varepsilon
}\left(  s\right)  +\overline{f}\left(  s\right)  \right\rangle +\left\langle
\widehat{Y}_{s},\widetilde{b}^{\varepsilon}\left(  s\right)  -\overline
{b}\left(  s\right)  \right\rangle \right] \\
& +\mathbb{E}\left[  \left\langle \widehat{Z}_{s}M_{s+1},\left(
\widetilde{\sigma}^{\varepsilon}\left(  s\right)  -\overline{\sigma}\left(
s\right)  \right)  M_{s+1}\right\rangle \right] \\
\geq & \alpha\mathbb{E}\left[  \sum\limits_{t=0}^{T}\left\vert \widehat{X}%
_{t}\right\vert ^{2}+\sum\limits_{t=0}^{T}\left\vert \widehat{Y}%
_{t}\right\vert ^{2}+\sum\limits_{t=0}^{T-1}\left\vert \widehat{Z}%
_{t}\widetilde{I}_{t}\right\vert ^{2}\right]  .
\end{array}
\label{f_c_c_2_difference_2}%
\end{equation}
On the other hand,%
\begin{align*}
&  \mathbb{E}\left\langle \widehat{X}_{s},-\widetilde{f}^{\varepsilon}\left(
s\right)  +\overline{f}\left(  s\right)  \right\rangle +\mathbb{E}\left\langle
\widehat{Y}_{s},\widetilde{b}^{\varepsilon}\left(  s\right)  -\overline
{b}\left(  s\right)  \right\rangle \\
&  \leq\frac{\alpha}{2}\mathbb{E}\left\vert \widehat{X}_{s}\right\vert
^{2}+\frac{1}{2\alpha}\mathbb{E}\left\vert \overline{f}\left(  s\right)
-\widetilde{f}^{\varepsilon}\left(  s\right)  \right\vert ^{2}+\frac{\alpha
}{2}\mathbb{E}\left\vert \widehat{Y}_{s}\right\vert ^{2}+\frac{1}{2\alpha
}\mathbb{E}\left\vert \widetilde{b}^{\varepsilon}\left(  s\right)
-\overline{b}\left(  s\right)  \right\vert ^{2}\\
&  \leq\frac{\alpha}{2}\mathbb{E}\left\vert \widehat{X}_{s}\right\vert
^{2}+\frac{\alpha}{2}\mathbb{E}\left\vert \widehat{Y}_{s}\right\vert
^{2}+\frac{C}{2\alpha}\varepsilon^{2}\mathbb{E}\left\vert \Delta v\right\vert
^{2}.
\end{align*}
and%
\begin{align*}
&  \mathbb{E}\left\langle \widehat{Z}_{s}M_{s+1},\left(  \widetilde{\sigma
}^{\varepsilon}\left(  s\right)  -\overline{\sigma}\left(  s\right)  \right)
M_{s+1}\right\rangle \\
&  \leq\frac{\alpha}{2C}\mathbb{E}\left\vert \widehat{Z}_{s}M_{s+1}\right\vert
^{2}+\frac{C}{2\alpha}\mathbb{E}\left\vert \left(  \widetilde{\sigma
}^{\varepsilon}\left(  s\right)  -\overline{\sigma}\left(  s\right)  \right)
M_{s+1}\right\vert ^{2}\\
&  \leq\frac{\alpha}{2}\mathbb{E}\left\vert \widehat{Z}_{s}\widetilde{I}%
\right\vert ^{2}+\frac{C}{2\alpha}\varepsilon^{2}\mathbb{E}\left\vert \Delta
v\right\vert ^{2}.
\end{align*}
Thus%
\begin{equation}%
\begin{array}
[c]{cl}
& \mathbb{E}\left[  \left\langle \widehat{X}_{s},-\widetilde{f}^{\varepsilon
}\left(  s\right)  +\overline{f}\left(  s\right)  \right\rangle +\left\langle
\widehat{Y}_{s},\widetilde{b}^{\varepsilon}\left(  s\right)  -\overline
{b}\left(  s\right)  \right\rangle \right] \\
& +\mathbb{E}\left\langle \widehat{Z}_{s}M_{s+1},\left(  \widetilde{\sigma
}^{\varepsilon}\left(  s\right)  -\overline{\sigma}\left(  s\right)  \right)
M_{s+1}\right\rangle \\
\leq & \frac{\alpha}{2}\mathbb{E}\left[  \left\vert \widehat{X}_{s}\right\vert
^{2}+\left\vert \widehat{Y}_{s}\right\vert ^{2}+\left\vert \widehat{Z}%
_{s}\widetilde{I}\right\vert ^{2}\right]  +C\varepsilon^{2}\mathbb{E}%
\left\vert \Delta v\right\vert ^{2}.
\end{array}
\label{f_c_c_2_difference_3}%
\end{equation}
Combining (\ref{f_c_c_2_difference_2}) and (\ref{f_c_c_2_difference_3}), we
have%
\[
\mathbb{E}\left[  \sum_{t=0}^{T}\left\vert \widehat{X}_{t}\right\vert
^{2}+\sum_{t=0}^{T}\left\vert \widehat{Y}_{t}\right\vert ^{2}+\sum_{t=0}%
^{T-1}\left\vert \widehat{Z}_{t}\widetilde{I}\right\vert ^{2}\right]  \leq
C\varepsilon^{2}\mathbb{E}\left\vert \Delta v\right\vert ^{2}.
\]
This completes the proof.
\end{proof}

Next we introduce the following variational equation:%
\begin{equation}
\left\{
\begin{array}
[c]{rcl}%
\Delta\xi_{t} & = & b_{x}\left(  t\right)  \xi_{t}+b_{y}\left(  t\right)
\eta_{t}+\zeta_{t}\widetilde{I}b_{\widetilde{z}}\left(  t\right)  +\delta
_{ts}b_{u}\left(  t\right)  \varepsilon\Delta v\\
&  & +\left[  \sigma_{x}\left(  t\right)  \xi_{t}+\sigma_{y}\left(  t\right)
\eta_{t}+\zeta_{t}\widetilde{I}\sigma_{\widetilde{z}}\left(  t\right)
+\delta_{ts}\varepsilon\left(  \Delta v\right)  ^{\ast}\sigma_{u}\left(
t\right)  \right]  M_{t+1},\\
\Delta\eta_{t} & = & -f_{x}\left(  t+1\right)  \xi_{t+1}-f_{y}\left(
t+1\right)  \eta_{t+1}-\zeta_{t+1}\widetilde{I}f_{\widetilde{z}}\left(
t+1\right) \\
&  & -\delta_{\left(  t+1\right)  s}f_{u}\left(  t+1\right)  \varepsilon\Delta
v+\zeta_{t}M_{t+1},\\
\xi_{0} & = & 0,\\
\eta_{T} & = & 0.
\end{array}
\right.  \label{f_c_c_2_variation}%
\end{equation}

By Assumption \ref{c5_assumption_m_pc} and Assumption
\ref{f_c_c_2_control_assumption}, when $t\in\left\{  1,...,T-1\right\}  $,%
\begin{equation}
\left(
\begin{array}
[c]{ccc}%
1 & 0 & 0\\
0 & 1 & 0\\
0 & 0 & \mathbb{E}\left[  M_{t+1}M_{t+1}^{\ast}|\mathcal{F}_{t}\right]
\end{array}
\right)  \left(
\begin{array}
[c]{ccc}%
-f_{x}\left(  t\right)  & -f_{y}\left(  t\right)  & -f_{\widetilde{z}}\left(
t\right) \\
b_{x}\left(  t\right)  & b_{y}\left(  t\right)  & b_{\widetilde{z}}\left(
t\right) \\
\sigma_{x}^{\ast}\left(  t\right)  & \sigma_{y}^{\ast}\left(  t\right)  &
\sigma_{\widetilde{z}}^{\ast}\left(  t\right)
\end{array}
\right)  \leq-\alpha\left(
\begin{array}
[c]{ccc}%
1 & 0 & 0\\
0 & 1 & 0\\
0 & 0 & I_{d-1}%
\end{array}
\right)  ; \label{f_c_c_2_monotone_1}%
\end{equation}

when $t=0$,%
\begin{equation}
\left(
\begin{array}
[c]{cc}%
b_{y}\left(  0\right)  & b_{\widetilde{z}}\left(  0\right) \\
\mathbb{E}\left[  M_{1}M_{1}^{\ast}|\mathcal{F}_{0}\right]  \sigma_{y}^{\ast
}\left(  0\right)  & \mathbb{E}\left[  M_{1}M_{1}^{\ast}|\mathcal{F}%
_{0}\right]  \sigma_{\widetilde{z}}^{\ast}\left(  0\right)
\end{array}
\right)  \leq-\alpha\left(
\begin{array}
[c]{cc}%
1 & 0\\
0 & I_{n-1}%
\end{array}
\right)  ; \label{f_c_c_2_monotone_2}%
\end{equation}

when $t=T$,%
\begin{equation}
-g_{x}\left(  T\right)  \leq-\alpha. \label{f_c_c_2_monotone_3}%
\end{equation}

Thus, the coefficients of\thinspace(\ref{f_c_c_2_variation}) satisfy the
monotone condition and there exists a unique solution $\left(  \xi,\eta
,\zeta\right)  $ to\thinspace(\ref{f_c_c_2_variation}). Similar to the proof
of Lemma \ref{f_c_c_2_xyz_approxi}, we have%
\begin{equation}
\mathbb{E}\left[  \sum_{t=0}^{T}\left\vert \xi_{t}\right\vert ^{2}+\sum
_{t=0}^{T}\left\vert \eta_{t}\right\vert ^{2}+\sum_{t=0}^{T-1}\left\vert
\zeta_{t}\widetilde{I}\right\vert ^{2}\right]  \leq C\varepsilon^{2}%
\mathbb{E}\left\vert \Delta v\right\vert ^{2}.
\label{f_c_c_2_variation_approxi}%
\end{equation}

Define%
\[%
\begin{array}
[c]{cr}%
\widetilde{\varphi}_{\mu}\left(  t\right)  = & \int_{0}^{1}\varphi_{\mu
}\left(  t,\bar{X}_{t}+\lambda\left(  X_{t}^{\varepsilon}-\bar{X}_{t}\right)
,\bar{Y}_{t}+\lambda\left(  Y_{t}^{\varepsilon}-\bar{Y}_{t}\right)  ,\bar
{Z}_{t}\widetilde{I}+\lambda\left(  Z_{t}^{\varepsilon}-\bar{Z}_{t}\right)
\widetilde{I},\right. \\
& \left.  \bar{u}_{t}+\lambda\left(  u_{t}^{\varepsilon}-\bar{u}_{t}\right)
\right)  d\lambda.
\end{array}
\]

where $\varphi=b$, $\sigma_{i}$,\ $f$, $l$, $h$ and $\mu=x$, $y$, $z$ and $u$.

\begin{lemma}
\label{f_c_c_2_variation_difference_approxi_0}Under Assumption
\ref{c5_assumption_m_pc} and Assumption \ref{f_c_c_2_control_assumption}, we
have%
\[
\mathbb{E}\left[  \sum_{t=0}^{T}\left\vert \widehat{X}_{t}-\xi_{t}\right\vert
^{2}+\sum_{t=0}^{T}\left\vert \widehat{Y}_{t}-\eta_{t}\right\vert ^{2}%
+\sum_{t=0}^{T-1}\left\vert \left(  \widehat{Z}_{t}-\zeta_{t}\right)
\widetilde{I}\right\vert ^{2}\right]  =o\left(  \varepsilon^{2}\right)  .
\]

\end{lemma}

\begin{proof}
Note that%
\begin{align*}
&  \varphi^{\varepsilon}\left(  t\right)  -\overline{\varphi}\left(  t\right)
\\
&  =\widetilde{\varphi}_{x}\left(  t\right)  \left(  X_{t}^{\varepsilon}%
-\bar{X}_{t}\right)  +\widetilde{\varphi}_{y}\left(  t\right)  \left(
Y_{t}^{\varepsilon}-\bar{Y}_{t}\right)  +\left(  Z_{t}^{\varepsilon}-\bar
{Z}_{t}\right)  \widetilde{I}\widetilde{\varphi}_{\widetilde{z}}\left(
t\right)  +\delta_{ts}\widetilde{\varphi}_{u}\left(  t\right)  \varepsilon
\Delta v.
\end{align*}
Set%
\[
\widetilde{X}_{t}=\widehat{X}_{t}-\xi_{t},\text{ }\widetilde{Y}_{t}%
=\widehat{Y}_{t}-\eta_{t},\text{ }\widetilde{Z}_{t}=\widehat{Z}_{t}-\zeta
_{t}.
\]
Then,%
\begin{equation}
\left\{
\begin{array}
[c]{rcl}%
\Delta\widetilde{X}_{t} & = & b_{x}\left(  t\right)  \widetilde{X}_{t}%
+b_{y}\left(  t\right)  \widetilde{Y}_{t}+\widetilde{Z}_{t}\widetilde{I}%
b_{\widetilde{z}}\left(  t\right)  +\Lambda_{1}\left(  t\right) \\
&  & +\left[  \sigma_{x}\left(  t\right)  \widetilde{X}_{t}+\sigma_{y}\left(
t\right)  \widetilde{Y}_{t}+\widetilde{Z}_{t}\widetilde{I}\sigma
_{\widetilde{z}}\left(  t\right)  +\Lambda_{2}\left(  t\right)  \right]
M_{t+1},\\
\Delta\widetilde{Y}_{t} & = & -f_{x}\left(  t+1\right)  \widetilde{X}%
_{t+1}-f_{y}\left(  t+1\right)  \widetilde{Y}_{t+1}-\widetilde{Z}%
_{t+1}\widetilde{I}f_{\widetilde{z}}\left(  t+1\right) \\
&  & -\Lambda_{3}\left(  t+1\right)  +\widetilde{Z}_{t}M_{t+1},\\
\widetilde{X}_{0} & = & 0,\\
\widetilde{Y}_{T} & = & 0,
\end{array}
\right.  \label{f_c_c_2_variation_difference_approxi}%
\end{equation}
where%
\begin{align*}
\Lambda_{1}\left(  t\right)   &  =\left(  \widetilde{b}_{x}\left(  t\right)
-b_{x}\left(  t\right)  \right)  \widehat{X}_{t}+\left(  \widetilde{b}%
_{y}\left(  t\right)  -b_{y}\left(  t\right)  \right)  \widehat{Y}_{t}\\
&  +\widehat{Z}_{t}\widetilde{I}\left(  \widetilde{b}_{\widetilde{z}}\left(
t\right)  -b_{\widetilde{z}}\left(  t\right)  \right)  +\delta_{ts}\left(
\widetilde{b}_{u}\left(  t\right)  -b_{u}\left(  t\right)  \right)
\varepsilon\Delta v,\\
\Lambda_{2}\left(  t\right)   &  =\left(  \widetilde{\sigma}_{x}\left(
t\right)  -\sigma_{x}\left(  t\right)  \right)  \widehat{X}_{t}+\left(
\widetilde{\sigma}_{y}\left(  t\right)  -\sigma_{y}\left(  t\right)  \right)
\widehat{Y}_{t}\\
&  +\widehat{Z}_{t}\widetilde{I}\left(  \widetilde{\sigma}_{\widetilde{z}%
}\left(  t\right)  -\sigma_{\widetilde{z}}\left(  t\right)  \right)
+\delta_{ts}\left(  \widetilde{\sigma}_{u}\left(  t\right)  -\sigma_{u}\left(
t\right)  \right)  \varepsilon\Delta v,\\
\Lambda_{3}\left(  t\right)   &  =-\left(  \widetilde{f}_{x}\left(  t\right)
-f_{x}\left(  t\right)  \right)  \widehat{X}_{t}-\left(  \widetilde{f}%
_{y}\left(  t\right)  -f_{y}\left(  t\right)  \right)  \widehat{Y}_{t}\\
&  -\widehat{Z}_{t}\widetilde{I}\left(  \widetilde{f}_{\widetilde{z}}\left(
t\right)  -f_{\widetilde{z}}\left(  t\right)  \right)  -\delta_{ts}\left(
\widetilde{f}_{u}\left(  t\right)  -f_{u}\left(  t\right)  \right)
\varepsilon\Delta v.
\end{align*}
According to\thinspace(\ref{f_c_c_2_variation_difference_approxi}),%
\[%
\begin{array}
[c]{rl}%
0= & \mathbb{E}\left\langle \widetilde{X}_{T},\widetilde{Y}_{T}\right\rangle
-\mathbb{E}\left\langle \widetilde{X}_{0},\widetilde{Y}_{0}\right\rangle \\
= & \mathbb{E}\sum_{t=0}^{T-1}\Delta\left\langle \widetilde{X}_{t}%
,\widetilde{Y}_{t}\right\rangle \\
= & \mathbb{E}\sum_{t=0}^{T}\left[  \left\langle \widetilde{X}_{t}%
,-\widetilde{\lambda}_{t}f_{\lambda}\left(  t\right)  \right\rangle
+\left\langle \widetilde{Y}_{t},\widetilde{\lambda}_{t}b_{\lambda}\left(
t\right)  \right\rangle +\left\langle \widetilde{Z}_{t},\widetilde{\lambda
}_{t}\sigma_{\lambda}\left(  t\right)  M_{t+1}M_{t+1}^{\ast}\right\rangle
\right] \\
& +\mathbb{E}\sum_{t=0}^{T}\left[  \left\langle \widetilde{X}_{t},-\Lambda
_{3}\left(  t\right)  \right\rangle +\left\langle \widetilde{Y}_{t}%
,\Lambda_{1}\left(  t\right)  \right\rangle +\left\langle \widetilde{Z}%
_{t}M_{t+1},\Lambda_{2}\left(  t\right)  M_{t+1}\right\rangle \right]  ,
\end{array}
\]
where%
\begin{align*}
\widetilde{\lambda}_{t}  &  =\left(  \widetilde{X}_{t},\widetilde{Y}%
_{t},\widetilde{Z}_{t}\widetilde{I}\right)  ,\\
b_{\lambda}\left(  t\right)   &  =\left(  b_{x}\left(  t\right)  ,b_{y}\left(
t\right)  ,b_{\widetilde{z}}\left(  t\right)  \right)  ,\\
\sigma_{\lambda}\left(  t\right)   &  =\left(  \sigma_{x}\left(  t\right)
,\sigma_{y}\left(  t\right)  ,\sigma_{\widetilde{z}}\left(  t\right)  \right)
,\\
f_{\lambda}\left(  t\right)   &  =\left(  f_{x}\left(  t\right)  ,f_{y}\left(
t\right)  ,f_{\widetilde{z}}\left(  t\right)  \right)  .
\end{align*}
Combining (\ref{f_c_c_2_monotone_1}), (\ref{f_c_c_2_monotone_2})
and\ (\ref{f_c_c_2_monotone_3}), we have%
\begin{equation}%
\begin{array}
[c]{cl}
& \mathbb{E}\sum_{t=0}^{T}\left[  \left\langle \widetilde{X}_{t},-\Lambda
_{3}\left(  t\right)  \right\rangle +\left\langle \widetilde{Y}_{t}%
,\Lambda_{1}\left(  t\right)  \right\rangle +\left\langle \widetilde{Z}%
_{t}M_{t+1},\Lambda_{2}\left(  t\right)  M_{t+1}\right\rangle \right] \\
\geq & \alpha\mathbb{E}\left[  \sum_{t=0}^{T}\left\vert \widetilde{X}%
_{t}\right\vert ^{2}+\sum_{t=0}^{T}\left\vert \widetilde{Y}_{t}\right\vert
^{2}+\sum_{t=0}^{T-1}\left\vert \widetilde{Z}_{t}\widetilde{I}\right\vert
^{2}\right]  .
\end{array}
\label{f_c_c_2_variation_difference_1}%
\end{equation}
Note that%
\[%
\begin{array}
[c]{cl}
& \mathbb{E}\left\langle \widetilde{X}_{t},-\Lambda_{3}\left(  t\right)
\right\rangle \\
= & \mathbb{E}\left\langle \widetilde{X}_{t},\left(  \widetilde{f}_{x}\left(
t\right)  -f_{x}\left(  t\right)  \right)  \widehat{X}_{t}\right\rangle
+\mathbb{E}\left\langle \widetilde{X}_{t},\left(  \widetilde{f}_{y}\left(
t\right)  -f_{y}\left(  t\right)  \right)  \widehat{Y}_{t}\right\rangle \\
& +\mathbb{E}\left\langle \widetilde{X}_{t},\widehat{Z}_{t}\widetilde{I}%
\left(  \widetilde{f}_{\widetilde{z}}\left(  t\right)  -f_{\widetilde{z}%
}\left(  t\right)  \right)  \right\rangle +\mathbb{E}\left\langle
\widetilde{X}_{t},\delta_{ts}\left(  \widetilde{f}_{u}\left(  t\right)
-f_{u}\left(  t\right)  \right)  \varepsilon\Delta v\right\rangle \\
\leq & \frac{\alpha}{2}\mathbb{E}\left\vert \widetilde{X}_{t}\right\vert
^{2}+\frac{2}{\alpha}\mathbb{E}\left[  \left\Vert \widetilde{f}_{x}\left(
t\right)  -f_{x}\left(  t\right)  \right\Vert ^{2}\left\vert \widehat{X}%
_{t}\right\vert ^{2}+\left\Vert \widetilde{f}_{y}\left(  t\right)
-f_{y}\left(  t\right)  \right\Vert ^{2}\left\vert \widehat{Y}_{t}\right\vert
^{2}\right] \\
& +\frac{2}{\alpha}\mathbb{E}\left[  \left\Vert \widetilde{f}_{\widetilde{z}%
}\left(  t\right)  -f_{\widetilde{z}}\left(  t\right)  \right\Vert
^{2}\left\vert \widehat{Z}_{t}\widetilde{I}\right\vert ^{2}+\delta
_{ts}\varepsilon^{2}\left\Vert \widetilde{f}_{u}\left(  t\right)
-f_{u}\left(  t\right)  \right\Vert ^{2}\left\vert \Delta v\right\vert
^{2}\right]  .
\end{array}
\]
When $\varepsilon\rightarrow0$, $\left\Vert \widetilde{f}_{\mu}\left(
t\right)  -f_{\mu}\left(  t\right)  \right\Vert \rightarrow0$\thinspace for
$\mu=x$, $y$, $\widetilde{z}$ and $u$. Then, by Lemma \thinspace
\ref{f_c_c_2_xyz_approxi},%
\[
\mathbb{E}\left\langle \widetilde{X}_{t},-\Lambda_{3}\left(  t\right)
\right\rangle \leq\frac{\alpha}{2}\mathbb{E}\left\vert \widetilde{X}%
_{t}\right\vert ^{2}+o\left(  \varepsilon^{2}\right)  .
\]
Similar results hold for the other terms in
(\ref{f_c_c_2_variation_difference_1}). Finally, we have%
\[
\mathbb{E}\left[  \sum_{t=0}^{T}\left\vert \widetilde{X}_{t}\right\vert
^{2}+\sum_{t=0}^{T}\left\vert \widetilde{Y}_{t}\right\vert ^{2}+\sum
_{t=0}^{T-1}\left\vert \widetilde{Z}_{t}\widetilde{I}\right\vert ^{2}\right]
\leq o\left(  \varepsilon^{2}\right)  .
\]
This completes the proof.
\end{proof}

By Lemma \ref{f_c_c_2_variation_difference_approxi_0}, we obtain%
\begin{align*}
&
\begin{array}
[c]{cc}
& J\left(  u^{\varepsilon}\left(  \cdot\right)  \right)  -J\left(  \bar
{u}\left(  \cdot\right)  \right)
\end{array}
\\
&
\begin{array}
[c]{cl}%
= & \mathbb{E}\sum_{t=0}^{T-1}\left[  l_{x}\left(  t\right)  \xi_{t}%
+l_{y}\left(  t\right)  \eta_{t}+\zeta_{t}\widetilde{I}l_{\widetilde{z}%
}\left(  t\right)  +\delta_{ts}l_{u}\left(  s\right)  \varepsilon\Delta
v\right] \\
& +\mathbb{E}\left[  h_{x}\left(  \bar{X}_{T}\right)  \xi_{T}\right]
+o\left(  \varepsilon\right)  .
\end{array}
\end{align*}

Introduce the following adjoint equation:%
\begin{equation}
\left\{
\begin{array}
[c]{rcl}%
\Delta p_{t} & = & -b_{x}\left(  t+1\right)  p_{t+1}-\sigma_{x}\left(
t+1\right)  \mathbb{E}\left[  M_{t+2}M_{t+2}^{\ast}|\mathcal{F}_{t+1}\right]
q_{t+1}^{\ast}\\
&  & +f_{x}\left(  t+1\right)  k_{t+1}+l_{x}\left(  t+1\right)  +q_{t}%
M_{t+1},\\
\Delta k_{t} & = & f_{y}\left(  t\right)  k_{t}-b_{y}\left(  t\right)
p_{t}-\sigma_{y}\left(  t\right)  \mathbb{E}\left[  M_{t+1}M_{t+1}^{\ast
}|\mathcal{F}_{t}\right]  q_{t}^{\ast}+l_{y}\left(  t\right) \\
&  & +\left[  \left(  f_{\widetilde{z}}^{\ast}\left(  t\right)  k_{t}%
-b_{\widetilde{z}}^{\ast}\left(  t\right)  p_{t}+l_{\widetilde{z}}^{\ast
}\left(  t\right)  \right)  \widetilde{I}^{\ast}\left(  \mathbb{E}\left[
M_{t+1}M_{t+1}^{\ast}|\mathcal{F}_{t}\right]  \right)  ^{\dag}\right]
M_{t+1}\\
&  & -\left(  q_{t}\mathbb{E}\left[  M_{t+1}M_{t+1}^{\ast}|\mathcal{F}%
_{t}\right]  \sigma_{\widetilde{z}}^{\ast}\left(  t\right)  \right)
\widetilde{I}^{\ast}\left(  \mathbb{E}\left[  M_{t+1}M_{t+1}^{\ast
}|\mathcal{F}_{t}\right]  \right)  ^{\dag}M_{t+1},\\
p_{T} & = & -h_{x}\left(  \bar{X}_{T}\right)  ,\\
k_{0} & = & 0.
\end{array}
\right.  \label{f_c_c_2_adjoint_eq}%
\end{equation}

Define the Hamiltonian function as follows:%
\[%
\begin{array}
[c]{r}%
H\left(  \omega,t,u,x,y,z,p,q,k\right)  =b\left(  \omega,t,x,y,z,u\right)
p+\sigma\left(  \omega,t,x,y,z,u\right)  \mathbb{E}\left[  M_{t+1}%
M_{t+1}^{\ast}|\mathcal{F}_{t}\right]  q^{\ast}\\
-f\left(  \omega,t,x,y,z,u\right)  k-l\left(  \omega,t,x,y,z,u\right)  .
\end{array}
\]

\begin{theorem}
Suppose that Assumption \ref{c5_assumption_m_pc} and Assumption
\ref{f_c_c_2_control_assumption} hold. Let $\bar{u}$ be an optimal control
for\ (\ref{f_c_c_2_state_eq})-(\ref{f_c_c_2_cost_eq}), $\left(  \bar{X}%
,\bar{Y},\bar{Z}\right)  $\ be the corresponding optimal trajectory and
$\left(  p,q,k\right)  $ be the solution to the adjoint
equation\ (\ref{f_c_c_2_adjoint_eq}). Then, for any $t\in\left\{
0,1,...,T\right\}  $, $\omega\in\Omega$\ and\ $v\in U_{t}$, we have%
\begin{equation}
\left\langle H_{u}\left(  \omega,t,\bar{u}_{t},\bar{X}_{t},\bar{Y}_{t},\bar
{Z}_{t},p_{t},q_{t},k_{t}\right)  ,v-\bar{u}_{t}\left(  \omega\right)
\right\rangle \leq0. \label{f_c_c_2_MP_result}%
\end{equation}

\end{theorem}

\begin{proof}
From the expression of $\xi_{t}$, $p_{t}$ for $t\in\left\{
0,1,...,T-1\right\}  $, we have%
\begin{align*}
&
\begin{array}
[c]{cl}
& \Delta\left\langle \xi_{t},p_{t}\right\rangle \\
= & \left\langle \xi_{t+1},\Delta p_{t}\right\rangle +\left\langle \Delta
\xi_{t},p_{t}\right\rangle
\end{array}
\\
&
\begin{array}
[c]{cl}%
= & \left\langle \xi_{t+1},-b_{x}\left(  t+1\right)  p_{t+1}-\sigma_{x}\left(
t+1\right)  \mathbb{E}\left[  M_{t+2}M_{t+2}^{\ast}|\mathcal{F}_{t+1}\right]
q_{t+1}^{\ast}+f_{x}\left(  t+1\right)  k_{t+1}+l_{x}\left(  t+1\right)
\right\rangle \\
& +\left\langle \left[  \sigma_{x}\left(  t\right)  \xi_{t}+\sigma_{y}\left(
t\right)  \eta_{t}+\zeta_{t}\widetilde{I}\sigma_{\widetilde{z}}\left(
t\right)  +\delta_{ts}\varepsilon\left(  \Delta v\right)  ^{\ast}\sigma
_{u}\left(  t\right)  \right]  M_{t+1},q_{t}M_{t+1}\right\rangle \\
& +\left\langle b_{x}\left(  t\right)  \xi_{t}+b_{y}\left(  t\right)  \eta
_{t}+\zeta_{t}\widetilde{I}b_{\widetilde{z}}\left(  t\right)  +\delta
_{ts}b_{u}\left(  t\right)  \varepsilon\Delta v,p_{t}\right\rangle +\Phi_{t},
\end{array}
\end{align*}
where%
\[%
\begin{array}
[c]{ccl}%
\Phi_{t} & = & \left\langle \xi_{t}+b_{x}\left(  t\right)  \xi_{t}%
+b_{y}\left(  t\right)  \eta_{t}+\zeta_{t}\widetilde{I}b_{\widetilde{z}%
}\left(  t\right)  +\delta_{ts}b_{u}\left(  t\right)  \varepsilon\Delta
v,q_{t}M_{t+1}\right\rangle \\
&  & +\left\langle \left[  \sigma_{x}\left(  t\right)  \xi_{t}+\sigma
_{y}\left(  t\right)  \eta_{t}+\zeta_{t}\widetilde{I}\sigma_{\widetilde{z}%
}\left(  t\right)  +\delta_{ts}\varepsilon\left(  \Delta v\right)  ^{\ast
}\sigma_{u}\left(  t\right)  \right]  M_{t+1},p_{t}\right\rangle
\end{array}
\]
We have $\mathbb{E}\left[  \Phi_{t}\right]  =0$. Besides,%
\begin{align*}
&  \mathbb{E}\left[  \left(  \sigma_{x}\left(  t\right)  \xi_{t}+\sigma
_{y}\left(  t\right)  \eta_{t}+\zeta_{t}\widetilde{I}\sigma_{\widetilde{z}%
}\left(  t\right)  +\delta_{ts}\varepsilon\left(  \Delta v\right)  ^{\ast
}\sigma_{u}\left(  t\right)  \right)  M_{t+1}q_{t}M_{t+1}\right] \\
&  =\mathbb{E}\left[  \left(  \xi_{t}\sigma_{x}\left(  t\right)  +\eta
_{t}\sigma_{y}\left(  t\right)  +\zeta_{t}\widetilde{I}\sigma_{\widetilde{z}%
}\left(  t\right)  +\delta_{ts}\varepsilon\left(  \Delta v\right)  ^{\ast
}\sigma_{u}\left(  t\right)  \right)  \mathbb{E}\left[  M_{t+1}M_{t+1}^{\ast
}|\mathcal{F}_{t}\right]  q_{t}^{\ast}\right]  .
\end{align*}
Similarly,%
\begin{align*}
&
\begin{array}
[c]{cl}
& \Delta\left\langle \eta_{t},k_{t}\right\rangle \\
= & \left\langle \Delta\eta_{t},k_{t+1}\right\rangle +\left\langle \eta
_{t},\Delta k_{t}\right\rangle
\end{array}
\\
&
\begin{array}
[c]{cl}%
= & \left\langle -f_{x}\left(  t+1\right)  \xi_{t+1}-f_{y}\left(  t+1\right)
\eta_{t+1}-\zeta_{t+1}\widetilde{I}f_{\widetilde{z}}\left(  t+1\right)
-\delta_{\left(  t+1\right)  s}f_{u}\left(  t+1\right)  \varepsilon\Delta
v,k_{t+1}\right\rangle \\
& +\left\langle \zeta_{t}M_{t+1},\left(  f_{\widetilde{z}}^{\ast}\left(
t\right)  k_{t}-b_{\widetilde{z}}^{\ast}\left(  t\right)  p_{t}%
+l_{\widetilde{z}}^{\ast}\left(  t\right)  \right)  \widetilde{I}^{\ast
}\left(  \mathbb{E}\left[  M_{t+1}M_{t+1}^{\ast}|\mathcal{F}_{t}\right]
\right)  ^{\dag}M_{t+1}\right\rangle \\
& -\left\langle \zeta_{t}M_{t+1},q_{t}\mathbb{E}\left[  M_{t+1}M_{t+1}^{\ast
}|\mathcal{F}_{t}\right]  \sigma_{\widetilde{z}}^{\ast}\left(  t\right)
\widetilde{I}^{\ast}\left(  \mathbb{E}\left[  M_{t+1}M_{t+1}^{\ast
}|\mathcal{F}_{t}\right]  \right)  ^{\dag}M_{t+1}\right\rangle \\
& +\left\langle \eta_{t},f_{y}\left(  t\right)  k_{t}-b_{y}\left(  t\right)
p_{t}-\sigma_{y}\left(  t\right)  \mathbb{E}\left[  M_{t+1}M_{t+1}^{\ast
}|\mathcal{F}_{t}\right]  q_{t}^{\ast}+l_{y}\left(  t\right)  \right\rangle
+\Psi_{t},
\end{array}
\end{align*}
where%
\[%
\begin{array}
[c]{ccl}%
\Psi_{t} & = & \left\langle \zeta_{t}M_{t+1},k_{t}+f_{y}\left(  t\right)
k_{t}-b_{y}\left(  t\right)  p_{t}-\sigma_{y}\left(  t\right)  \mathbb{E}%
\left[  M_{t+1}M_{t+1}^{\ast}|\mathcal{F}_{t}\right]  q_{t}^{\ast}%
+l_{y}\left(  t\right)  \right\rangle \\
&  & +\left\langle \eta_{t},\left(  f_{\widetilde{z}}^{\ast}\left(  t\right)
k_{t}-b_{\widetilde{z}}^{\ast}\left(  t\right)  p_{t}+l_{\widetilde{z}}^{\ast
}\left(  t\right)  \right)  \widetilde{I}^{\ast}\left(  \mathbb{E}\left[
M_{t+1}M_{t+1}^{\ast}|\mathcal{F}_{t}\right]  \right)  ^{\dag}M_{t+1}%
\right\rangle .\\
&  & -\left\langle \eta_{t},q_{t}\mathbb{E}\left[  M_{t+1}M_{t+1}^{\ast
}|\mathcal{F}_{t}\right]  \sigma_{\widetilde{z}}^{\ast}\left(  t\right)
\widetilde{I}^{\ast}\left(  \mathbb{E}\left[  M_{t+1}M_{t+1}^{\ast
}|\mathcal{F}_{t}\right]  \right)  ^{\dag}M_{t+1}\right\rangle .
\end{array}
\]
Furthermore,%
\[%
\begin{array}
[c]{cl}
& \mathbb{E}\left[  \zeta_{t}M_{t+1}\left(  f_{\widetilde{z}}^{\ast}\left(
t\right)  k_{t}-b_{\widetilde{z}}^{\ast}\left(  t\right)  p_{t}%
+l_{\widetilde{z}}^{\ast}\left(  t\right)  \right)  \widetilde{I}^{\ast
}\left(  \mathbb{E}\left[  M_{t+1}M_{t+1}^{\ast}|\mathcal{F}_{t}\right]
\right)  ^{\dag}M_{t+1}\right] \\
& -\mathbb{E}\left[  \zeta_{t}M_{t+1}q_{t}\mathbb{E}\left[  M_{t+1}%
M_{t+1}^{\ast}|\mathcal{F}_{t}\right]  \sigma_{\widetilde{z}}^{\ast}\left(
t\right)  \widetilde{I}^{\ast}\left(  \mathbb{E}\left[  M_{t+1}M_{t+1}^{\ast
}|\mathcal{F}_{t}\right]  \right)  ^{\dag}M_{t+1}\right] \\
= & \mathbb{E}\left[  \zeta_{t}M_{t+1}M_{t+1}^{\ast}\left(  \mathbb{E}\left[
M_{t+1}M_{t+1}^{\ast}|\mathcal{F}_{t}\right]  \right)  ^{\dag}\widetilde{I}%
\left(  f_{\widetilde{z}}\left(  t\right)  k_{t}-b_{\widetilde{z}}\left(
t\right)  p_{t}+l_{\widetilde{z}}\left(  t\right)  \right)  \right] \\
& -\mathbb{E}\left[  \zeta_{t}M_{t+1}M_{t+1}^{\ast}\left(  \mathbb{E}\left[
M_{t+1}M_{t+1}^{\ast}|\mathcal{F}_{t}\right]  \right)  ^{\dag}\widetilde{I}%
\sigma_{\widetilde{z}}\left(  t\right)  \mathbb{E}\left[  M_{t+1}M_{t+1}%
^{\ast}|\mathcal{F}_{t}\right]  q_{t}^{\ast}\right] \\
= & \mathbb{E}\left[  \zeta_{t}\widetilde{I}\left(  f_{\widetilde{z}}\left(
t\right)  k_{t}-b_{\widetilde{z}}\left(  t\right)  p_{t}+l_{\widetilde{z}%
}\left(  t\right)  -\sigma_{\widetilde{z}}\left(  t\right)  \mathbb{E}\left[
M_{t+1}M_{t+1}^{\ast}|\mathcal{F}_{t}\right]  q_{t}^{\ast}\right)  \right]  .
\end{array}
\]

Then, we obtain%
\[%
\begin{array}
[c]{cl}
& \mathbb{E}\left[  \Delta\left(  \xi_{t}p_{t}+\eta_{t}k_{t}\right)  \right]
\\
= & \mathbb{E}\left[  -\xi_{t+1}b_{x}\left(  t+1\right)  p_{t+1}+\xi_{t}%
b_{x}\left(  t\right)  p_{t}\right. \\
& -\xi_{t+1}\sigma_{x}\left(  t+1\right)  \mathbb{E}\left[  M_{t+2}%
M_{t+2}^{\ast}|\mathcal{F}_{t+1}\right]  q_{t+1}^{\ast}+\left\langle \xi
_{t},\sigma_{x}\left(  t\right)  \mathbb{E}\left[  M_{t+1}M_{t+1}^{\ast
}|\mathcal{F}_{t}\right]  q_{t}^{\ast}\right\rangle \\
& -\eta_{t+1}f_{y}\left(  t+1\right)  k_{t+1}+\eta_{t}f_{y}\left(  t\right)
k_{t}\\
& -\zeta_{t+1}\widetilde{I}f_{\widetilde{z}}\left(  t+1\right)  k_{t+1}%
+\zeta_{t}\widetilde{I}f_{\widetilde{z}}\left(  t\right)  k_{t}\\
& +\xi_{t+1}l_{x}\left(  t+1\right)  +\eta_{t}l_{y}\left(  t\right)
+\zeta_{t}\widetilde{I}l_{\widetilde{z}}\left(  t\right) \\
& +\varepsilon\delta_{ts}\left\langle b_{u}^{\ast}\left(  t\right)
p_{t},\Delta v\right\rangle +\varepsilon\delta_{ts}\left\langle \sigma
_{u}\left(  t\right)  \mathbb{E}\left[  M_{t+1}M_{t+1}^{\ast}|\mathcal{F}%
_{t}\right]  q_{t}^{\ast},\Delta v\right\rangle \\
& -\varepsilon\delta_{\left(  t+1\right)  s}\left\langle f_{u}^{\ast}\left(
t+1\right)  k_{t+1},\Delta v\right\rangle .
\end{array}
\]
Therefore,%
\[%
\begin{array}
[c]{cl}
& -\mathbb{E}\left[  h_{x}\left(  \bar{X}_{T}\right)  \xi_{T}\right] \\
= & \mathbb{E}\left[  \left\langle \xi_{T},p_{T}\right\rangle +\left\langle
\eta_{T},k_{T}\right\rangle -\left\langle \xi_{0},p_{0}\right\rangle
-\left\langle \eta_{0},k_{0}\right\rangle \right] \\
= & \sum_{t=0}^{T-1}\mathbb{E}\Delta\left(  \left\langle \xi_{t}%
,p_{t}\right\rangle +\left\langle \eta_{t},k_{t}\right\rangle \right) \\
= & \mathbb{E}\left[  b_{x}\left(  0\right)  \xi_{0}p_{0}+\xi_{0}\sigma
_{x}\left(  0\right)  \mathbb{E}\left[  M_{1}M_{1}^{\ast}|\mathcal{F}%
_{0}\right]  q_{0}^{\ast}+\eta_{0}f_{y}\left(  0\right)  k_{0}+\zeta
_{0}\widetilde{I}f_{\widetilde{z}}\left(  0\right)  k_{0}\right] \\
& +\sum_{t=0}^{T-1}\mathbb{E}\left[  l_{x}\left(  t\right)  \xi_{t}%
+l_{y}\left(  t\right)  \eta_{t}+\zeta_{t}\widetilde{I}l_{\widetilde{z}%
}\left(  t\right)  \right] \\
& +\sum_{t=0}^{T}\delta_{ts}\varepsilon\mathbb{E}\left[  \left\langle
b_{u}^{\ast}\left(  t\right)  p_{t},\Delta v\right\rangle +\left\langle
\sigma_{u}\left(  t\right)  \mathbb{E}\left[  M_{t+1}M_{t+1}^{\ast
}|\mathcal{F}_{t}\right]  q_{t}^{\ast},\Delta v\right\rangle -\left\langle
f_{u}^{\ast}\left(  t\right)  k_{t},\Delta v\right\rangle \right]  .
\end{array}
\]
Notice that $\xi_{0}=0$, $k_{0}=0$. So%
\[%
\begin{array}
[c]{cl}
& \mathbb{E}\sum_{t=0}^{T-1}\left[  l_{x}\left(  t\right)  \xi_{t}%
+l_{y}\left(  t\right)  \eta_{t}+\zeta_{t}\widetilde{I}l_{\widetilde{z}%
}\left(  t\right)  \right]  +\mathbb{E}\left[  h_{x}\left(  \bar{X}%
_{T}\right)  \xi_{T}\right] \\
= & -\varepsilon\mathbb{E}\left[  \left\langle b_{u}^{\ast}\left(  s\right)
p_{s}+\sigma_{u}\left(  s\right)  \mathbb{E}\left[  M_{s+1}M_{s+1}^{\ast
}|\mathcal{F}_{t}\right]  q_{s}^{\ast}-f_{u}^{\ast}\left(  s\right)
k_{s},\Delta v\right\rangle \right]  .
\end{array}
\]
Since $\lim_{\varepsilon\rightarrow0}\frac{1}{\varepsilon}\left[  J\left(
u^{\varepsilon}\left(  \cdot\right)  \right)  -J\left(  \bar{u}\left(
\cdot\right)  \right)  \right]  \geq0$, we obtain%
\[
\mathbb{E}\left[  \left\langle b_{u}^{\ast}\left(  s\right)  p_{s}+\sigma
_{u}\left(  s\right)  \mathbb{E}\left[  M_{s+1}M_{s+1}^{\ast}|\mathcal{F}%
_{t}\right]  q_{s}^{\ast}-f_{u}^{\ast}\left(  s\right)  k_{s}-l_{u}^{\ast
}\left(  s\right)  ,\Delta v\right\rangle \right]  \leq0.
\]
Then, (\ref{f_c_c_2_MP_result}) holds due to that $s$ is taking arbitrarily.
This completes the proof.
\end{proof}


\begin{thebibliography}{99}                                                                                               %






\bibitem {b82}Bensoussan, A. (1982). Lectures on stochastic control. In
Nonlinear filtering and stochastic control (pp. 1-62). Springer, Berlin, Heidelberg.

\bibitem {bcc15}Bielecki, T. R., Cialenco, I., \& Chen, T. (2015). Dynamic
conic finance via backward stochastic difference equations. SIAM Journal on
Financial Mathematics, 6(1), 1068-1122.



\bibitem {b78}Bismut, J. M. (1978). An introductory approach to duality in
optimal stochastic control. SIAM review, 20(1), 62-78.

\bibitem {ce08}Cohen, S. N., \& Elliott, R. J. (2008). Solutions of backward
stochastic differential equations on Markov chains. Communications on
stochastic analysis, 2(2), 251-262.

\bibitem {ce10+}Cohen, S. N., \& Elliott, R. J. (2010). A general theory of
finite state backward stochastic difference equations. Stochastic Processes
and their Applications, 120(4), 442-466.

\bibitem {ce11}Cohen, S. N., \& Elliott, R. J. (2011). Backward stochastic
difference equations and nearly time-consistent nonlinear expectations. SIAM
Journal on Control and Optimization, 49(1), 125-139.



\bibitem {dz99}Dokuchaev, N., \& Zhou, X. Y. (1999). Stochastic controls with
terminal contingent conditions. Journal of Mathematical Analysis and
Applications, 238(1), 143-165.

\bibitem {em11}Eberlein, E., Gehrig, T., \& Madan, D. B. (2011). Pricing to
acceptability: With applications to valuing one's own credit risk.

\bibitem {eh97}El Karoui, N., \& Huang, S. J. (1997). A general result of
existence and uniqueness of backward stochastic differential equations. Pitman
Research Notes in Mathematics Series, 27-38.

\bibitem {hjx18}Hu, M., Ji, S., \& Xue, X. (2018). A global stochastic maximum
principle for fully coupled forward-backward stochastic systems. arXiv
preprint arXiv:1803.02109.

\bibitem {hjx18+}Hu, M., Ji, S., \& Xue, X. (2018). A note on the global
stochastic maximum principle for fully coupled forward-backward stochastic
systems. arXiv preprint arXiv:1812.10469.



\bibitem {ji-liu}Ji, S. \& Liu, H. (2018). Fully coupled finite state
forward-backward stochastic difference equations, arXiv.

\bibitem {k72}Kushner, H. J. (1972). Necessary conditions for continuous
parameter stochastic optimization problems. SIAM Journal on Control, 10(3), 550-565.

\bibitem {lz01}Lim, A. E., \& Zhou, X. Y. (2001). Linear-quadratic control of
backward stochastic differential equations. SIAM journal on control and
optimization, 40(2), 450-474.

\bibitem {ly14}Lin, Y., \& Yang, H. (2014). Discrete-Time BSDEs with Random
Terminal Horizon. Stochastic Analysis and Applications, 32(1), 110-127.

\bibitem {lz15}Lin, X., \& Zhang, W. (2015). A maximum principle for optimal
control of discrete-time stochastic systems with multiplicative noise. IEEE
Transactions on Automatic Control, 60(4), 1121-1126.

\bibitem {m10}Madan, D. B. (2010). Conserving capital by adjusting deltas for
gamma in the presence of skewness. Journal of Risk and Financial Management,
3(1), 1-25.

\bibitem {p90}Peng, S. (1990). A general stochastic maximum principle for
optimal control problems. SIAM Journal on control and optimization, 28(4), 966-979.

\bibitem {p93}Peng, S. (1993). Backward stochastic differential equations and
applications to optimal control. Applied Mathematics and Optimization, 27(2), 125-144.

\bibitem {pp90}Pardoux, E., \& Peng, S. (1990). Adapted solution of a backward
stochastic differential equation. Systems \& Control Letters, 14(1), 55-61.

\bibitem {ss99}Schroder, M., \& Skiadas, C. (1999). Optimal consumption and
portfolio selection with stochastic differential utility. Journal of Economic
Theory, 89(1), 68-126.

\bibitem {sz13}Shi, Y., \& Zhu, Q. (2013). Partially observed optimal controls
of forward-backward doubly stochastic systems. ESAIM: Control, Optimisation
and Calculus of Variations, 19(3), 828-843.

\bibitem {s10}Stadje, M. (2010). Extending dynamic convex risk measures from
discrete time to continuous time: A convergence approach. Insurance:
Mathematics and Economics, 47(3), 391-404.

\bibitem {w09}Williams, N. (2009). On dynamic principal-agent problems in
continuous time. University of Wisconsin, Madison.

\bibitem {w98}Wu, Z. (1998). Maximum principle for optimal control problem of
fully coupled forward-backward stochastic systems. Systems Science and
Mathematical sciences, 3, 249-259.

\bibitem {x95}Xu, W. (1995). Stochastic maximum principle for optimal control
problem of forward and backward system. The ANZIAM Journal, 37(2), 172-185.

\bibitem {y10}Yong, J. (2010). Optimality variational principle for controlled
forward-backward stochastic differential equations with mixed initial-terminal
conditions. SIAM Journal on Control and Optimization, 48(6), 4119-4156.

\bibitem {zs11}Zhang, L., \& Shi, Y. (2011). Maximum principle for
forward-backward doubly stochastic control systems and applications. ESAIM:
Control, Optimisation and Calculus of Variations, 17(4), 1174-1197.


\end{thebibliography}
\end{document}